\documentclass{amsart}

\usepackage{amssymb}
\usepackage{amsfonts}
\usepackage{amsmath}
\usepackage{graphicx}
\usepackage{mathrsfs}
\usepackage{dsfont}
\usepackage{amscd}
\usepackage{multirow}
\usepackage[all]{xy}
\normalfont
\usepackage[T1]{fontenc}
\usepackage{calligra}
\usepackage{verbatim}
\usepackage[usenames,dvipsnames]{color}
\usepackage[colorlinks=true,linkcolor=Blue,citecolor=Violet]{hyperref}

\newcommand{\N}{{\mathds{N}}}

\newcommand{\R}{{\mathds{R}}}
\newcommand{\C}{{\mathds{C}}}

\newcommand{\D}{{\mathfrak{D}}}
\newcommand{\A}{{\mathfrak{A}}}
\newcommand{\B}{{\mathfrak{B}}}

\newcommand{\Lip}{{\mathsf{L}}}
\newcommand{\Hilbert}{{\mathscr{H}}}

\newcommand{\dist}{{\mathsf{dist}}}

\newcommand{\propinquity}[1]{{\mathsf{\Lambda}^\ast_{#1}}}

\newcommand{\Kantorovich}[1]{{\mathsf{mk}_{#1}}}

\newcommand{\Haus}[1]{{\mathsf{Haus}_{#1}}}

\newcommand{\StateSpace}{{\mathscr{S}}}

\newcommand{\mongekant}{{Mon\-ge-Kan\-to\-ro\-vich metric}}

\newcommand{\Lqcms}{{\JLL} quantum compact metric space}
\newcommand{\gQqcms}{quasi-Leibniz quantum compact metric space}
\newcommand{\Qqcms}[1]{${#1}$--quasi-Leibniz quantum compact metric space}

\newcommand{\unit}{1}

\newcommand{\sa}[1]{{\mathfrak{sa}\left({#1}\right)}}

\newcommand{\JLL}{Lei\-bniz}

\newcommand{\dom}[1]{{\operatorname*{dom}({#1})}}
\newcommand{\diam}[2]{{\mathrm{diam}\left({#1},{#2}\right)}}
\newcommand{\covn}[3]{{\mathrm{cov}_{{#1}}\left({#2}\middle\vert{#3}\right)}}

\newcommand{\tunnelset}[3]{{\text{\calligra Tunnels}\,\left[\left({#1}\right)\stackrel{#3}{\longrightarrow}\left({#2}\right)\right]}}

\newcommand{\bridgelength}[2]{{\lambda\left({#1}\middle|{#2}\right)}}

\newcommand{\Jordan}[2]{{{#1}\circ{#2}}} 
\newcommand{\Lie}[2]{{\left\{{#1},{#2}\right\}}} 

\newcommand{\targetsettunnel}[3]{{\mathfrak{t}_{#1}\left({#2}\middle\vert{#3}\right)}}

\newcommand{\tunneldepth}[1]{{\delta\left(#1\right)}}
\newcommand{\tunnelreach}[1]{{\rho\left(#1\right)}}
\newcommand{\tunnellength}[1]{{\lambda\left(#1\right)}}

\newcommand{\tunnelextent}[1]{{\chi\left({#1}\right)}}

\newcommand{\co}[1]{{\overline{\mathrm{co}}\left(#1\right)}}
\newcommand{\alg}[1]{{\mathfrak{#1}}}

\theoremstyle{plain}
\newtheorem{theorem}{Theorem}[section]

\newtheorem{step}{Step}

\newtheorem{corollary}[theorem]{Corollary}

\newtheorem{lemma}[theorem]{Lemma}
\newtheorem{proposition}[theorem]{Proposition}

\newtheorem{theorem-definition}[theorem]{Theorem-Definition}

\theoremstyle{definition}
\newtheorem{definition}[theorem]{Definition}

\newtheorem{notation}[theorem]{Notation}

\newtheorem{convention}[theorem]{Convention}

\theoremstyle{remark}

\newtheorem{remark}[theorem]{Remark}

\renewcommand{\geq}{\geqslant}
\renewcommand{\leq}{\leqslant}

\numberwithin{equation}{section}

\begin{document}

\title{A Compactness Theorem for The Dual Gromov-Hausdorff Propinquity}
\author{Fr\'{e}d\'{e}ric Latr\'{e}moli\`{e}re}
\email{frederic@math.du.edu}
\urladdr{http://www.math.du.edu/\symbol{126}frederic}
\address{Department of Mathematics \\ University of Denver \\ Denver CO 80208}

\date{\today}

\begin{abstract}
We prove a compactness theorem for the dual Gromov-Hausdorff propinquity as a noncommutative analogue of the Gromov compactness theorem for the {Gro\-mov-\-Haus\-dorff} distance. Our theorem is valid for subclasses of {\gQqcms s} of the closure of finite dimensional {\gQqcms s} for the dual propinquity. While finding characterizations of this class proves delicate, we show that all nuclear, quasi-diagonal {\gQqcms s} are limits of finite dimensional {\gQqcms s}. This result involves a mild extension of the definition of the dual propinquity to {\gQqcms s}, which is presented in the first part of this paper.
\end{abstract}
\maketitle

\tableofcontents


\section{Introduction}

The dual Gromov-Hausdorff propinquity \cite{Latremoliere13b,Latremoliere13c,Latremoliere14} is an analogue of the {Gro\-mov-\-Haus\-dorff} distance \cite{Gromov81} defined on a class of noncommutative algebras, called {\Lqcms s}, seen as a generalization of the algebras of Lipschitz  functions over metric spaces. The dual propinquity is designed to provide a framework to extend techniques from metric geometry \cite{Gromov} to noncommutative geometry \cite{Connes}. In this paper, we prove a generalization of Gromov's compactness theorem to the dual propinquity, and study the related issue of finite dimensional approximations for {\Lqcms s}. As a consequence of our study, we extend the dual Gromov-Hausdorff propinquity to various classes of {\gQqcms s}, which are generalized forms of noncommutative Lipschitz algebras. Indeed, we prove that nuclear quasi-diagonal {\Lqcms s} are limits of finite dimensional {\gQqcms s}, yet we do not know whether they are limits of finite dimensional {\Lqcms s}.

An important property of the Gromov-Hausdorff distance is a characterization of compact classes of compact metric spaces \cite{Gromov81}:

\begin{theorem}[Gromov's Compactness Theorem]\label{Gromov-Compactness-thm}
A class $\mathcal{S}$ of compact metric spaces is totally bounded for the Gromov-Hausdorff distance if, and only if the following two assertions hold:
\begin{enumerate}
\item there exists $D\geq 0$ such that for all $(X,\mathrm{m}) \in \mathcal{S}$, the diameter of $(X,\mathrm{m})$ is less or equal to $D$,
\item there exists a function $G: (0,\infty)\rightarrow \N$ such that for every $(X,\mathrm{m}) \in \mathcal{S}$, and for every $\varepsilon > 0$, the smallest number $\mathrm{Cov}_{(X,\mathrm{m})}(\varepsilon)$ of balls of radius $\varepsilon$ needed to cover $(X,\mathrm{m})$ is no more than $G(\varepsilon)$. 
\end{enumerate}
Since the Gromov-Hausdorff distance is complete, a class of compact metric spaces is compact for the Gromov-Hausdorff distance if and only if it is closed and totally bounded.
\end{theorem}

Condition (2) in Theorem (\ref{Gromov-Compactness-thm}) is meaningful since any compact metric space is totally bounded, which is also equivalent to the statement that compact metric spaces can be approximated by their finite subsets for the Hausdorff distance. Thus, intimately related to the question of extending Theorem (\ref{Gromov-Compactness-thm}) to the dual propinquity, is the question of finite dimensional approximations of {\Lqcms s}. However, this latter question proves delicate. In order to explore this issue, we are led to work within larger classes of generalized Lipschitz algebras, which we name the {\gQqcms s}. The advantage of these classes is that it is possible to extend to them the dual propinquity, and then prove that a large class of {\Lqcms s} are limits of finite dimensional {\gQqcms s} for the dual propinquity. In particular, the closure of many classes of finite dimensional {\gQqcms s} contain all \emph{nuclear quasi-diagonal} {\Lqcms s}. The various possible classes of {\gQqcms s} we introduce represent various degrees of departure from the original Leibniz inequality. 

The dual propinquity is our answer to the challenge raised by the quest for a well-behaved analogue of the Gromov-Hausdorff distance designed to work within the C*-algebraic framework in noncommutative metric geometry \cite{Connes89, Connes, Rieffel98a, Rieffel99, Rieffel00, Rieffel01, Rieffel02, Rieffel05, Rieffel06, Rieffel08, Rieffel09, Rieffel10, Rieffel10c, Rieffel11, Rieffel12, Rieffel14}. Recent research in noncommutative metric geometry suggests, in particular, that one needs a strong tie between the quantum topological structure, provided by a C*-algebra, and the quantum metric structure, if one wishes to study the behavior of C*-algebraic related structures such as projective modules under metric convergence. The quantum metric structure over a C*-algebra $\A$ is given by a seminorm $\Lip$ defined on a dense subspace $\dom{\Lip}$ of the self-adjoint part of $\A$ such that:
\begin{enumerate}
\item $\{a\in\dom{\Lip} : \Lip(a) = 0 \} = \R \unit_\A$,
\item the distance on the state space $\StateSpace(\A)$ of $\A$ defined, for any two $\varphi,\psi\in\StateSpace(\A)$ by:
\begin{equation}\label{mongekant-eq}
\Kantorovich{\Lip}(\varphi,\psi) = \sup\left\{ |\varphi(a) - \psi(a)| : a\in\dom{\Lip}\text{ and }\Lip(a)\leq 1 \right\}
\end{equation}
metrizes the weak* topology on $\StateSpace(\A)$.
\end{enumerate}
Such a seminorm is called a Lip-norm, and the metric defined by Expression (\ref{mongekant-eq}) is called the {\mongekant}, by analogy with the classical picture \cite{Kantorovich40,Kantorovich58}. In particular, it appears that Lip-norms should satisfy a form of the Leibniz inequality. In \cite{Latremoliere13} and onward, we thus connected quantum topology and quantum metric by adding to Lip-norms $\Lip$ on unital C*-algebras $\A$ the requirement that:
\begin{equation}\label{Leibniz-eq1}
\Lip\left(\frac{ab+ba}{2}\right) \leq \Lip(a)\|b\|_\A + \Lip(b)\|a\|_\A
\end{equation}
and
\begin{equation}\label{Leibniz-eq2}
\Lip\left(\frac{ab-ba}{2i}\right) \leq \Lip(a)\|b\|_\A + \Lip(b)\|a\|_\A\text{,}
\end{equation}
where $\|\cdot\|_\A$ is the norm of the underlying C*-algebra $\A$.

Yet, the quantum Gromov-Hausdorff distance, introduced by Rieffel \cite{Rieffel00} as the first noncommutative analogue of the Gromov-Hausdorff distance, did not capture the C*-algebraic structure, as illustrated with the fact that the distance between two non-isomorphic C*-algebras could be null. In order to strengthen Rieffel's distance to make *-isomorphism necessary, one may consider one of at least two general approaches. 

A first idea is to modify the quantum Gromov-Hausdorff distance so that it captures some quantum topological aspects of the underlying quantum metric spaces, while not tying the metric and topological structure together. For instance, Kerr's distance \cite{kerr02} replaces the state space in Expression (\ref{mongekant-eq}) with spaces of unital completely positive matrix-valued linear maps, which capture some additional topological information, while the notion of Lip-norm, i.e. metric structure, is essentially unchanged and in particular, does not involve the Leibniz properties (\ref{Leibniz-eq1}), (\ref{Leibniz-eq2}). This first approach is shared, in some fashion or another, by all early attempts to fix the weakness of the coincidence axiom \cite{kerr02,li03,li06,kerr09}. However, Kerr introduced in \cite{kerr02} a weakened notion of the Leibniz property, called the $F$-Leibniz property, motivated by the nontrivial question of the completeness of his matricial distance. Specifically, when applying Kerr's construction specifically to a class of $F$-Leibniz quantum compact metric spaces, Kerr proved in \cite[Theorem 5.3]{kerr02} that the distance thus obtained (which is not the matricial distance, but rather the so-called $F$-Leibniz complete distance \cite[Section 5]{kerr02}) is complete. The original matricial distance of Kerr is not complete in general, as seen in \cite{Guido06}. Thus, the F-Leibniz property is a natural requirement, as far as Kerr's distance is concerned, and as well for our own dual propinquity, as we shall see in this paper.

Yet, as with Rieffel's construction, Kerr's distance is not known to satisfy the triangle inequality once restricted to a specific class of $F$-Leibniz compact quantum metric spaces. The dual propinquity was designed precisely to handle this situation.

A second approach is to tie together the metric and topological structure of quantum metric spaces before attempting to define a Gromov-Hausdorff distance. This approach involves working on a more restrictive category of quantum compact metric spaces. One then realizes that quite a few challenges arise when trying to define an analogue of the Gromov-Hausdorff distance. These challenges owe to the fact that the definition of an analogue of the Gromov-Hausdorff distance involves a form of embedding of two compact quantum metric spaces into some other space, and various properties of this analogue, such as the triangle inequality, become harder to establish when one puts strong constraints on the possible embedding. The Leibniz property of Expressions (\ref{Leibniz-eq1}) and (\ref{Leibniz-eq2}) are examples of such strong constraints. First steps in this direction can be found in Rieffel's quantum proximity \cite{Rieffel10c}, and the restriction of Kerr's distance to $F$-Leibniz compact quantum metric spaces \cite{kerr02}, which are not known to even be a pseudo-metric, as the proof of the triangle inequality is elusive.

The dual propinquity follows the second approach, and owes its name to Rieffel's proximity. Its construction answered the rather long-standing challenge to employ the second approach described above toward a successful conclusion. The dual propinquity is defined on {\Lqcms s}, which are compact quantum metric spaces whose Lip-norms are defined on dense Jordan-Lie sub-algebras of the self-adjoint part of the underlying C*-algebras, and which satisfy the Leibniz inequalities (\ref{Leibniz-eq1}) and (\ref{Leibniz-eq2}). All the main examples of compact quantum metric spaces are in fact {\Lqcms s} \cite{Rieffel00,Rieffel01,Rieffel02,Ozawa05,li05}. Now, the dual propinquity is a complete metric on the class of {\Lqcms s}: in particular, distance zero implies *-isomorphism, in addition to isometry of quantum metric structures, and the triangle inequality is satisfied. Several examples of convergence for the dual propinquity are known \cite{Latremoliere05,Latremoliere13b,Rieffel11}.

Moreover, several stronger ties between quantum metric and quantum topology have been proposed, most notably Rieffel's strong Leibniz property and Rieffel's compact C*-metric spaces \cite{Rieffel10c}, both of which are special cases of {\Lqcms s}. The dual propinquity can be specialized to these classes, in the sense that its construction may, if desired, only involve quantum metric spaces in these classes. We also note that the notion of {\Lqcms} can be extended to the framework of quantum locally compact metric spaces \cite{Latremoliere05b,Latremoliere12b,Latremoliere14b}. 

The problem of determining which {\Lqcms s} is a limit of finite dimensional {\Lqcms s} for the dual propinquity, however, challenges us in this paper to explore a somewhat relaxed form of the Leibniz inequality, while keeping, informally, the same tie between quantum topology and quantum metric structure. Indeed, while any compact quantum metric space is within an arbitrarily small quantum Gromov-Hausdorff distance to some finite dimensional quantum metric space \cite[Theorem 13.1]{Rieffel00}, Rieffel's construction of these finite dimensional approximations does not produce finite dimensional C*-algebras nor Leibniz Lip-norms. In fact, the construction of finite dimensional approximations for {\Lqcms s} using the dual propinquity remains elusive in general. If $(\A,\Lip)$ is a {\Lqcms s}, we seek a sequence $(\B_n,\Lip_n)_{n\in\N}$ of finite dimensional {\Lqcms s} which converge to $(\A,\Lip)$ for the dual propinquity. The first question is: what is the source of the C*-algebras $\B_n$ ?

A natural approach to the study of this problem is to first seek C*-algebras which naturally come with finite dimensional C*-algebra approximations in a quantum topological sense. Nuclearity and quasi-diagonality, for instance, provide such approximations. The next natural question becomes: how to equip the finite dimensional topological approximations of some C*-algebra $\A$ with quantum metric structures, given a Lip-norm on $\A$? Our effort led us to an answer in this paper, if we allow a bit of flexibility. When working with a nuclear quasi-diagonal {\Lqcms} $(\A,\Lip)$, we can equip finite dimensional approximations with Lip-norms which are not necessarily Leibniz, but satisfy a slightly generalized form of the Leibniz identity, in such a way as to obtain a metric approximation for the dual propinquity. This weakened form of the Leibniz property is referred to as the quasi-Leibniz property. Informally, one may require that the deficiency in the Leibniz property for approximations of nuclear quasi-diagonal {\Lqcms s} be arbitrarily small, though not null.

We emphasize that a crucial difficulty which we address in our work regards the Leibniz, and more generally, quasi-Leibniz property of the Lip-norms. Without a specific interest in the Leibniz property, one may obtain finite dimensional approximations for various classes of compact quantum metric spaces, typically under some finite dimensional topological approximation properties such as nuclearity, exactness, and others, as seen for instance in \cite[Proposition 3.9, 3.10, 3.11]{kerr02}, \cite[Theorem 5.3, Theorem 5.8]{kerr09}. The challenge, however, is to work with a \emph{metric} on classes of \emph{Leibniz} or \emph{quasi-Leibniz} compact metric spaces for which such finite dimensional approximations are valid, meaning in part that the Lip-norms on the finite dimensional approximating spaces must also satisfy a form of Leibniz identity. These matters turn out to add quite a bit of complexity to the problem.

Our paper thus begins with the extension of the dual Gromov-Hausdorff propinquity to new classes of quantum compact metric spaces which we call {\gQqcms s}. We prove, using our techniques developed in \cite{Latremoliere13c,Latremoliere14}, that the extended propinquity retains all of its basic properties: it is a complete metric on various classes of {\gQqcms}. An important aspect of this construction is that one first picks a class of {\gQqcms s}, and then within this class, the dual propinquity gives a metric --- while involving only spaces from one's chosen class. In practice, the purpose of this construction is to allow one to chose, informally, the least general class of {\gQqcms s}, and ideally even the class of {\Lqcms s}, since this would imply that all the intermediate spaces involved in the computation of the dual propinquity would possess as strong a Leibniz property as possible, and thus allow for more amenable computations. 

 We then prove our compactness theorem for {\gQqcms s} in the next section of our paper. We take advantage of our extension of the dual propinquity to prove our theorem for a large collection of {\gQqcms s}. This theorem includes a statement about the original class of {\Lqcms s}. The proof of our theorem is split in two main steps: we first investigate the compactness of some classes of finite dimensional {\gQqcms s}. We then prove our compactness theorem for classes of {\gQqcms s} which are within the closure of certain classes of finite dimensional {\gQqcms s}.

We then conclude our paper with the proof that nuclear-quasi-diagonal {\Lqcms s}  (and even many nuclear quasi-diagonal {\gQqcms s}) are within the closure of many classes of \emph{finite dimensional} {\gQqcms s}, to which our compactness theorem would thus apply.

\section{Quasi-Leibniz quantum compact metric spaces and the Dual Gromov-Hausdorff Propinquity}

The framework for our paper is a class of compact quantum metric spaces constructed over C*-algebras and whose Lip-norms are well-behaved with respect to the multiplication. The desirable setup is to require the Leibniz property \cite{Rieffel09, Rieffel10c, Latremoliere13,Latremoliere13c,Latremoliere14}. Yet, as we shall see in the second half of this work, the Leibniz property is difficult to obtain for certain finite-dimensional approximations. We are thus led to a more flexible framework, although we purposefully wish to stay, informally, close to the original Leibniz property, while accommodating the constructions of finite dimensional approximations in the last section of our paper, and potential future examples. It could also be noted that certain constructions in noncommutative geometry, such as twisted spectral triples \cite{Connes08}, would lead to seminorms which are not satisfying the Leibniz inequality, yet would fit within our new framework. With this in mind, we propose the following as the basic objects of our study:

\begin{notation}
The norm of any normed vector space $X$ is denoted by $\|\cdot\|_X$.
\end{notation}

\begin{notation}
Let $\A$ be a unital C*-algebra. The unit of $\A$ is denoted by $\unit_\A$. The subspace of the self-adjoint elements in $\A$ is denoted by $\sa{\A}$. The state space of $\A$ is denoted by $\StateSpace(\A)$.
\end{notation}

\begin{notation}
Let $\A$ be a C*-algebra. The \emph{Jordan product} of $a,b \in \sa{\A}$ is the element $\Jordan{a}{b} = \frac{1}{2}\left(ab + ba\right)$ and the \emph{Lie product} of $a,b \in \sa{\A}$ is the element $\Lie{a}{b} = \frac{1}{2i}\left(ab - ba\right)$.
\end{notation}

Our generalization of the Leibniz property will rely on:
\begin{definition}\label{permissible-def}
We endow $[0,\infty)^4$ with the following order:
\begin{multline*}
\forall x=(x_1,x_2,x_3,x_4), y=(y_1,y_2,y_3,y_4) \\
x\preccurlyeq y \iff \left(\forall j\in\{1,\ldots,4\} \quad x_j\leq y_j\right)\text{.}
\end{multline*}
A function $F : [0,\infty)^4\rightarrow [0,\infty)$ is \emph{permissible} when:
\begin{enumerate}
\item $F$ is nondecreasing from $([0,\infty)^4,\preccurlyeq)$ to $([0,\infty),\leq)$,
\item for all $x,y,l_x,l_y \in [0,\infty)$ we have:
\begin{equation}\label{permissible-eq}
x l_y + y l_x \leq F(x,y,l_x,l_y)\text{.}
\end{equation}
\end{enumerate}
\end{definition}

\begin{definition}\label{QLeibniz-def}
Let $\A$ be a C*-algebra and $F$ be a permissible function. A seminorm $\Lip$ defined on a Jordan-Lie sub-algebra $\dom{\Lip}$ of $\sa{\A}$ is \emph{$F$-quasi-Leibniz} when, for all $a,b \in \dom{\Lip}$:
\begin{align}\label{qLeibniz-def}
\Lip\left(\Jordan{a}{b}\right) &\leq F(\|a\|_\A, \|b\|_\A, \Lip(a), \Lip(b)) \\
\intertext{and}
\Lip\left(\Lie{a}{b}\right) &\leq F(\|a\|_\A,\|b\|_\A,\Lip(a),\Lip(b)) \text{.}
\end{align}
\end{definition}

The order on $[0,\infty)^4$ used in Definition (\ref{permissible-def}) was used in \cite{kerr02} to define $F$-Leibniz seminorms. Thus Definition (\ref{QLeibniz-def}) is a special case of the definition of $F$-Leibniz seminorms: the key difference is that our notion of permissibility for a function $F$ also guarantees that the maximum of an $F$-quasi-Leibniz seminorm and a Leibniz seminorm is again an $F$-quasi-Leibniz seminorm, i.e. the function $F$ is not modified by this construction. This will prove essential in our core theorems for this section: Theorem (\ref{tunnel-composition-thm}), which is the basis for the proof that the dual propinquity satisfies the triangle inequality, and Theorem (\ref{prop-thm}), which shows that the extension of the dual propinquity to {\gQqcms s} satisfies all the desired properties.

The notion of an $F$-quasi-Leibniz seminorm includes various examples of seminorms from differential calculi \cite{Blackadar91} or from twisted spectral triples with non-isometric twists.

We are now ready to define the main objects for our work:

\begin{definition}\label{Qqcms-def}
A pair $(\A,\Lip)$ of a unital C*-algebra $\A$  and a seminorm $\Lip$ defined on a dense Jordan-Lie subalgebra $\dom{\Lip}$ of $\sa{\A}$ is a \emph{\Qqcms{F}} for some permissible function $F$ when:
\begin{enumerate}
\item $\left\{a \in \dom{\Lip} : \Lip(a) = 0 \right\} = \R\unit_\A$,
\item $\{a\in\sa{\A}:\Lip(a)\leq 1\}$ is closed in $\|\cdot\|_\A$,
\item the \emph{\mongekant} on $\StateSpace(\A)$, defined for any two states $\varphi,\psi \in \StateSpace(\A)$ by:
\begin{equation}\label{MongeKant-def}
\Kantorovich{\Lip}(\varphi, \psi) = \sup \left\{ |\varphi(a) - \psi(a)| : a\in\dom{\Lip}\text{ and }\Lip(a) \leq 1 \right\}\text{,}
\end{equation}
induces the weak* topology on $\StateSpace(\A)$,
\item $\Lip$ is $F$-quasi-Leibniz on $\dom{\Lip}$.
\end{enumerate}
The seminorm $\Lip$ of a {\Qqcms{F}} is called an \emph{$F$-quasi-Leibniz Lip-norm}.
\end{definition}

The following two special cases of {\gQqcms} will be central to our current work. Our original construction of the dual propinquity held for {\Lqcms s}:

\begin{definition}
A \emph{\Lqcms} $(\A,\Lip)$ is a {\Qqcms{L}} where the function $L$ is given by:
\begin{equation*}
L : (x,y,l_x,l_y) \in [0,\infty)^4 \longmapsto x l_y + y l_x \text{.}
\end{equation*}
\end{definition}

Our main motivation in extending the domain of the propinquity is to accommodate the following type of {\gQqcms s}, as they will be the type of quantum compact metric spaces which we will construct as finite dimensional approximations to nuclear quasi-diagonal {\Lqcms s}:

\begin{notation}\label{CDQqcms-notation}
A \emph{{\Qqcms{(C,D)}}}, for some $C \geq 1$ and $D \geq 0$, is a {\Qqcms{F_{C,D}}} where the function $F_{C,D}$ is defined by:
\begin{equation*}
F_{C,D} : x,y,l_x,l_y \in [0,\infty)^4 \longmapsto C(x l_y + y l_x) + D l_x l_y \text{.}
\end{equation*}
\end{notation}

Thus, if $(\A,\Lip)$ is a {\Qqcms{(C,D)}}, then for all $a,b\in\dom{\Lip}$, we have:
\begin{equation*}
\Lip\left(\Jordan{a}{b}\right) \leq C\left(\Lip(a)\|b\|_\A + \|a\|_\A \Lip(b)\right) + D \Lip(a)\Lip(b)
\end{equation*}
and
\begin{equation*}
\Lip\left(\Lie{a}{b}\right) \leq C\left(\Lip(a)\|b\|_\A + \|a\|_\A \Lip(b)\right) + D \Lip(a)\Lip(b)\text{,}
\end{equation*}
and {\Lqcms s} are exactly the {\Qqcms{(1,0)}s}.

We note that a pair $(\A,\Lip)$ with $\Lip$ a seminorm defined on a dense subspace of $\sa{\A}$ and satisfying Assertion (1) of Definition (\ref{Qqcms-def}) is a Lipschitz pair, and a \emph{Lipschitz pair} which satisfies Assertion (3) of Definition (\ref{Qqcms-def}) is a \emph{compact quantum metric space}, while $\Lip$ is then known as a Lip-norm \cite{Rieffel98a, Rieffel99}.

We use the following convention in this paper:
\begin{convention}
If $\Lip$ is a seminorm defined on a dense subset $\dom{\Lip}$ of some vector space $X$ then we extend $\Lip$ to $X$ by setting $\Lip(a) = \infty$ for all $a\in X\setminus\dom{\Lip}$, and we check that if $\Lip$ is closed, i.e $\{a\in\sa{\A}:\Lip(a)\leq 1\}$ is closed, then its extension is lower-semicontinuous with respect to the norm. We adopt the convention that $0\cdot\infty = 0$ when working with seminorms. Moreover, if $F$ is a permissible function, then we implicitly extend $F$ to a function from $[0,\infty]^4$ to $[0,\infty]$ by setting $F(x,y,l_x,l_y) = \infty$ if $\infty \in \{x,y,l_x,l_y\}$ whenever necessary.
\end{convention}

We remark that if $(\A,\Lip)$ is a {\Qqcms{(C,D)}} then, since $\Lip(\unit_\A) = 0$ and thus $\Lip(a) = \Lip(a\unit_\A)\leq C\Lip(a)$ for all $a\in\sa{\A}$, we must have $C\geq 1$, as given in Notation (\ref{CDQqcms-notation}). Of course, $C\geq 1$ and $D\geq 0$ are necessary for $F_{C,D}$ to be permissible as well; we see that in this case this condition does not remove any generality.

\begin{remark}
Of course, if $(\A,\Lip)$ is a {\Qqcms{(C,D)}} for some $C\geq 1$, $D\geq 0$, then it is a {\Qqcms{(C',D')}} for any $C'\geq C$ and $D'\geq D$. More generally, if $G\geq F$ are two nondecreasing nonnegative functions over $[0,\infty)^4$, and if $(\A,\Lip)$ is a {\Qqcms{F}}, then $(\A,\Lip)$ is a {\Qqcms{G}} as well.
\end{remark}

We make a simple remark, which will prove useful when proving the triangle inequality for our extension of the dual propinquity:
\begin{remark}\label{max-qlip-norms-rmk}
If $\Lip_1$ and $\Lip_2$ are, respectively, $F_1$ and $F_2$-quasi-Leibniz Lip-norms on some unital C*-algebra $\A$, then $\max\{\Lip_1,\Lip_2\}$ is a $\max\{F_1,F_2\}$-quasi-Leibniz Lip-norm on $\A$.

In particular, if $\Lip_1$ and $\Lip_2$ are respectively $(C_1,D_1)$ and $(C_2,D_2)$-quasi-Leibniz Lip-norms on some C*-algebra $\A$, then $\max\{\Lip_1,\Lip_2\}$ is a \begin{equation*}\left(\max\{C_1,C_2\},\max\{D_1,D_2\}\right)\end{equation*} quasi-Leibniz Lip-norm on $\A$.
\end{remark}

A morphism $\pi : (\A,\Lip_\A)\rightarrow (\B,\Lip_\B)$ is easily defined as a unital *-morphism from $\A$ to $\B$ whose dual map, restricted to the state space $\StateSpace(\B)$, is a Lipschitz map from $(\StateSpace(\B),\Kantorovich{\Lip_\B})$ to $(\StateSpace(\A),\Kantorovich{\Lip_\A})$. In this manner, {\gQqcms s} form a category. The notion of isomorphism in this category is the noncommutative generalization of bi-Lipschitz maps. However, as is usual with metric spaces, a more constrained notion of isomorphism is given by isometries.

\begin{definition}
An isometric isomorphsism $\pi : (\A,\Lip_\A)\rightarrow (\B,\Lip_\B)$ between two {\gQqcms s} $(\A,\Lip_\A)$ and $(\B,\Lip_\B)$ is a unital *-isomorphism from $\A$ to $\B$ whose dual map is an isometry from $(\StateSpace(\B),\Kantorovich{\Lip_\B})$ onto $(\StateSpace(\A),\Kantorovich{\Lip_\A})$.
\end{definition}

The assumption that Lip-norms are lower semi-continuous in Definition (\ref{Qqcms-def}) allows for the following useful characterizations of isometric isomorphisms:
\begin{theorem}[Theorem 6.2, \cite{Rieffel00}]
Let $(\A,\Lip_\A)$ and $(\B,\Lip_\B)$ be two {\gQqcms s}. A *-isomorphism $\pi : \A\rightarrow\B$ is an isometric isomorphism if and only if $\Lip_\B\circ\pi = \Lip_\A$.
\end{theorem}

The fundamental characterization of quantum compact metric spaces, due to Rieffel \cite{Rieffel98a, Rieffel99, Ozawa05}, will be used repeatedly in this work, so we include it here:

\begin{theorem}[Rieffel's characterization of compact quantum metric spaces]\label{Rieffel-thm}
Let $(\A,\Lip)$ be a pair of a unital C*-algebra $\A$ and a seminorm $\Lip$ defined on a dense subspace of $\sa{\A}$ such that $\{a\in\sa{\A}:\Lip(a) = 0 \} = \R\unit_\A$. The following assertions are equivalent:
\begin{enumerate}
\item The {\mongekant} $\Kantorovich{\Lip}$ metrizes the weak* topology of $\StateSpace(\A)$,
\item $\diam{\StateSpace(\A)}{\Kantorovich{\Lip}} < \infty$ and $\{a\in\sa{\A} : \Lip(a)\leq 1, \|a\|_\A\leq 1\}$ is totally bounded for $\|\cdot\|_\A$,
\item for some state $\mu \in\StateSpace(\A)$, the set $\{a\in\sa{\A} : \mu(a) = 0 \text{ and }\Lip(a)\leq 1\}$ is totally bounded for $\|\cdot\|_\A$,
\item for all states  $\mu \in\StateSpace(\A)$, the set $\{a\in\sa{\A} : \mu(a) = 0 \text{ and }\Lip(a)\leq 1\}$ is totally bounded for $\|\cdot\|_\A$.
\end{enumerate}
\end{theorem}

One may replace total boundedness by actual compactness in Theorem (\ref{Rieffel-thm}) when working with {\gQqcms s}, since their Lip-norms are lower semi-continuous.

We extend the theory of the dual Gromov-Hausdorff propinquity to the class of {\gQqcms s}. Our original construction \cite{Latremoliere13c} was developed for {\Lqcms s}. Much of our work carries naturally to {\gQqcms s}. We follow the improvements to our original construction made in \cite{Latremoliere14} for our general construction, and we refer to these two works heavily, focusing here only on the changes which are needed.

We begin by extending the notion of a tunnel between {\gQqcms s}, which is our noncommutative analogue of a pair of isometric embeddings of two {\gQqcms s} into a third {\gQqcms}.

\begin{definition}\label{tunnel-def}
Let $F : [0,\infty)^4\rightarrow [0,\infty)$ be a permissible function. Let $(\A_1,\Lip_1)$, $(\A_2,\Lip_2)$ be two {\Qqcms{F}s}. An \emph{$F$-tunnel} $(\D,\Lip_\D,\pi_1,\pi_2)$ from $(\A_1,\Lip_1)$ to $(\A_2,\Lip_2)$ is a {\Qqcms{F}} $(\D,\Lip_\D)$ and two *-epimorphisms $\pi_1:\D \twoheadrightarrow \A_1$ and $\pi_2 : \D \twoheadrightarrow \A_2$ such that, for all $j\in\{1,2\}$, and for all $a\in\sa{\A_j}$, we have:
\begin{equation*}
\Lip_j(a) = \inf \{ \Lip_\D(d) : d\in\sa{\D}\text{ and }\pi_j(d) = a \}\text{.}
\end{equation*}
\end{definition}

\begin{notation}
For any $C\geq 1$ and $D\geq 0$ and any two {\Qqcms{(C,D)}s} $(\A,\Lip_\A)$ and $(\B,\Lip_\B)$, a \emph{$(C,D)$-tunnel} from $(\A,\Lip_\A)$ to $(\B,\Lip_\B)$ is a $F_{C,D}$-tunnel, where $F_{C,D} : (a,b,l_a,l_b) \in [0,\infty)^4\mapsto C(a l_b + b l_a) + D l_a l_b$. 
\end{notation}

Tunnels always exists between any two {\gQqcms s}, as in \cite[Proposition 4.6]{Latremoliere13}. We will include this construction later on for the benefit of the reader, and then we will be able to use this construction to provide a bound on the propinquity.

We assign a numerical value to a tunnel designed to measure how far two {\gQqcms s} are, or rather how long the given tunnel is. As a matter of notation:
\begin{notation}
The Hausdorff distance on compact subsets of a compact metric space $(X,\mathrm{m})$ is denoted by $\Haus{\mathrm{m}}$. If $N$ is a norm on a vector space $X$, the Hausdorff distance on the class of closed subsets of any compact subset of $X$ is denoted by $\Haus{N}$.
\end{notation}

\begin{notation}
For any positive unital linear map $\pi : \A\rightarrow\B$ between two unital C*-algebras $\A$ and $\B$, the map $\varphi\in\StateSpace(\B) \mapsto \varphi\circ\pi \in\StateSpace(\A)$ is denote by $\pi^\ast$.
\end{notation}

\begin{definition}\label{extent-def}
Let $(\A_1,\Lip_1)$ and $(\A_2,\Lip_2)$ be two {\Qqcms{F}s} for some permissible function $F:[0,\infty)^4\rightarrow [0,\infty)$. The \emph{extent} $\tunnelextent{\tau}$ of an $F$-tunnel $\tau = (\D,\Lip_\D,\pi_1,\pi_2)$ from $(\A_1,\Lip_1)$ to $(\A_2,\Lip_2)$ is the real number:
\begin{equation*}
\max \left\{ \Haus{\Kantorovich{\Lip_\D}}(\StateSpace(\D),\pi_j^\ast(\StateSpace(\A_j))) : j \in \{1,2\} \right\}\text{.}
\end{equation*}
\end{definition}

We now see that tunnels always exist:

\begin{proposition}\label{tunnels-exist-prop}
Let $F$ be a permissible function and let $(\A,\Lip_\A)$ and $(\B,\Lip_\B)$ be two {\Qqcms{F}s}. Let:
\begin{equation*}
D = \max\left\{\diam{\StateSpace(\A)}{\Kantorovich{\Lip_\A}}, \diam{\StateSpace(\B)}{\Kantorovich{\Lip_\B}} \right\}\text{.}
\end{equation*}
Then for all $\epsilon \geq 0$ such that $D+\epsilon > 0$, there exists an $F$-tunnel $\tau$ with:
\begin{equation*}
\tunnelextent{\tau} \leq D + \epsilon \text{.}
\end{equation*}
\end{proposition}

\begin{proof}
We follow our earlier work in \cite[Proposition 4.6]{Latremoliere13}. 

Let $\Hilbert$ be a separable, infinite dimensional Hilbert space. Since $\A$ and $\B$ are separable (as the weak* topologies on their state spaces are metrizable), they both admit respective faithful unital representations $\pi_\A$ and $\pi_\B$ on $\Hilbert$. Let $\Omega$ be the C*-algebra of all bounded linear operators on $\Hilbert$. 

We now define an $F$-tunnel as follows. Let $\epsilon \geq 0$ be chosen so that $D + \epsilon > 0$. For any $a \in \sa{\A}$ and $b\in\sa{\B}$, we set:
\begin{equation*}
\Lip(a,b) = \max\left\{ \Lip_\A(a), \Lip_\B(b), \frac{1}{D+\epsilon} \|\pi_\A(a) - \pi_\B(b)\|_\Omega \right\} \text{.}
\end{equation*} 

First we observe that:
\begin{equation*}
\begin{split}
\Lip\left(\Jordan{a}{a'}, \Jordan{b}{b'}\right) &\leq \max\{\Lip_\A(\Jordan{a}{a'}), \Lip_\B(\Jordan{b}{b'}), \frac{1}{D+\epsilon}\|\pi_\A(aa')-\pi_\B(bb')\|_\Omega\}\\
&\leq \max\left\{\begin{array}{l} 
F(\|a\|_\A,\|a'\|_\A,\Lip_\A(a),\Lip_\A(a')),\\
F(\|b\|_\B,\|b'\|_\B,\Lip_\B(b),\Lip_\B(b')),\\
\frac{1}{D+\epsilon} \left(\|a\|_\A\|\pi_\A(a')-\pi_\B(b')\|_\Omega + \|b'\|_\B\|\pi_\A(a)-\pi_\B(b)\|_\Omega\right)
\end{array}\right\}\text{.}
\end{split}
\end{equation*}
Now, since $F$ is permissible, we note that:
\begin{equation*}
F(\|a\|_\A,\|a'\|_\A,\Lip_\A(a),\Lip_\A(a'))\leq F(\|(a,b)\|_{\A\oplus\B} , \|(a',b')\|_{\A\oplus\B}, \Lip(a,b), \Lip(a',b'))
\end{equation*}
while similarly:
\begin{equation*}
F(\|b\|_\B,\|b'\|_\B,\Lip_\B(b),\Lip_\B(b'))\leq F(\|(a,b)\|_{\A\oplus\B} , \|(a',b')\|_{\A\oplus\B}, \Lip(a,b), \Lip(a',b'))
\end{equation*}
and
\begin{multline*}
  \|a\|_\A\frac{\|\pi_\A(a')-\pi_\B(b')\|_\Omega}{D+\epsilon} + \|b'\|_\B\frac{\|\pi_\A(a)-\pi_\B(b)\|_\Omega}{D+\epsilon} \\ \leq F(\|(a,b)\|_{\A\oplus\B} , \|(a',b')\|_{\A\oplus\B}, \Lip(a,b), \Lip(a',b'))\text{.}
\end{multline*}
The same computation holds for the Lie product. Thus in conclusion, $\Lip$ is $F$-quasi-Leibniz.

Now, it is easy to check that $(\A\oplus\B,\Lip, \rho_\A,\rho_\B)$ is an $F$-tunnel, where $\rho_\A : \A\oplus\B \twoheadrightarrow \A$ and $\rho_\B: \A\oplus\B \twoheadrightarrow\B$ are the canonical surjections.

Indeed, let $a\in\sa{\A}$ with $\Lip_\A(a)<\infty$ and fix $\psi\in\StateSpace(\A)$. Then for any state $\varphi\in\StateSpace(\A)$:
\begin{equation*}
|\varphi(a-\psi(a)\unit_\A)| = |\varphi(a) - \psi(a)| \leq D \Lip_\A(a) \text{.}
\end{equation*}

Hence, we have:
\begin{equation*}
\Lip(a,\psi(a)\unit_\B) = \Lip_\A(a)
\end{equation*}
since $\Lip_\B(\psi(a)\unit_\B) = 0$. Since $\Lip_\A(a) \leq \Lip(a,b)$ for any $b\in\sa{\B}$, this shows that $\Lip_\A$ is the quotient of $\Lip$ for $\rho_\A$.

The same reasoning holds when $\A$ and $\B$ are switched. In addition, $\Lip$ is a Lip-norm on $\A\oplus\B$ by \cite{Rieffel00}. Thus we have proven that $(\A\oplus\B,\Lip,\rho_\A,\rho_\B)$ is an $F$-tunnel from $(\A,\Lip_\A)$ to $(\B,\Lip_\B)$.

The computation of the extent follows our work in \cite{Latremoliere13b}. Let $\varphi\in\StateSpace(\A)$. Since $\pi_\A$ is faithful, $\varphi$ defines a state on $\pi_\A(\A)$, which we then extend by Hahn-Banach to a state $\mu$ of $\Omega$. Let $\psi = \mu\circ\pi_\B$ and note that $\psi\in\StateSpace(\B)$.

Now, if $(a,b) \in \sa{\A\oplus\B}$ with $\Lip(a,b)\leq 1$, then:
\begin{equation*}
|\varphi\circ\rho_\A(a,b)-\psi\circ\rho_\B(a,b)| = |\mu(\pi_\A(a)) - \mu(\pi_\B(b))| \leq \|\pi_\A(a)-\pi_\B(b)\|_\Omega \leq D+\epsilon\text{.} 
\end{equation*}

Now, any state $\nu \in \StateSpace(\A\oplus\B)$ is a convex combination $t\varphi\circ\rho_\A + (1-t)\theta\circ\rho_\B$ of a state $\varphi$ in $\StateSpace(\A)$ and a state $\theta$ in $\StateSpace(\B)$ (by the bipolar theorem), for some $t\in [0,1]$. Using the above notations, we first find $\psi\in\StateSpace(\B)$ such that $\Kantorovich{\Lip}(\varphi\circ\rho_\A,\psi\circ\rho_\B) \leq D+\epsilon$, then we use the fact that $\Kantorovich{\Lip}$ is convex in each of its variable to see that:
\begin{equation*}
\Kantorovich{\Lip}(t \varphi \circ\rho_\A + (1-t)\theta\circ\rho_\B, t\psi\circ\rho_\B + (1-t)\theta\circ\rho_\B) \leq D + \epsilon\text{.}
\end{equation*}

The reasoning above is symmetric in $\A$ and $\B$, hence our proposition is proven.
\end{proof}

The dual propinquity may be specialized to various subclasses of {\Lqcms s}. Somewhat different requirements can be made on the class of allowable tunnels; for our purpose we shall follow the notion of an appropriate collection of tunnels, as in \cite{Latremoliere14}. We begin with the following observation, which makes explicit use of Assertion (\ref{permissible-eq}) of Definition (\ref{permissible-def}).

\begin{theorem}\label{tunnel-composition-thm}
Let $F$ be a permissible function. Let $(\A,\Lip_\A)$, $(\B,\Lip_\B)$ and $(\alg{C},\Lip_{\alg{C}})$ be three {\Qqcms{F}s}. Let $\tau_1 = (\D_1,\Lip_1,\pi_1,\pi_2)$ be an $F$-tunnel from $(\A,\Lip_\A)$ to $(\B,\Lip_\B)$ and $\tau_2=(\D_2,\Lip_2,\rho_1,\rho_2)$ be an $F$-tunnel from $(\B,\Lip_\B)$ to $(\alg{C},\Lip_{\alg{C}})$. Let $\varepsilon > 0$.

If, for all $(d_1,d_2)\in\sa{\D_1\oplus\D_2}$, we set:
\begin{equation*}
\Lip(d_1,d_2) = \max\left\{ \Lip_1(d_1), \Lip_2(d_2), \frac{1}{\varepsilon}\left\|\pi_2(d_1) - \rho_1(d_2) \right\|_\B \right\}\text{,}
\end{equation*}
and if $\eta_1 : (d_1,d_2)\in \D_1\oplus\D_2 \mapsto \pi_1(d_1)$ and $\eta_2:(d_1,d_2)\in \D_1\oplus\D_2 \mapsto \rho_2(d_2)$, then $\tau_3 = (\D_1\oplus\D_2,\Lip, \eta_1,\eta_2)$ is an $F$-tunnel from $(\A,\Lip_\A)$ to $(\alg{C},\Lip_{\alg{C}})$, whose extent satisfies:
\begin{equation*}
\tunnelextent{\tau_3} \leq \tunnelextent{\tau_1} + \tunnelextent{\tau_2} + \varepsilon\text{.}
\end{equation*}
\end{theorem}

\begin{proof}
Let $\dom{\Lip_1} = \{d\in\sa{\D_1}:\Lip_1(d)<\infty\}$ and $\dom{\Lip_2}=\{d\in\sa{\D_2}:\Lip_2(d)<\infty\}$. Note that:
\begin{equation*}
\dom{\Lip} = \{ (d_1,d_2)\in\sa{\D_1\oplus\D_2} : \Lip(d_1,d_2)<\infty \} = \dom{\Lip_1}\oplus\dom{\Lip_2}\text{.}
\end{equation*}

Let $N(d_1,d_2) = \frac{1}{\varepsilon}\|\pi_2(d_1)-\rho_1(d_2)\|_\B$ for all $(d_1,d_2)\in\dom{\Lip}$. For all $d_1,d_1'\in\sa{\D_1}$ and $d_2,d_2'\in\sa{\D_2}$, we have:
\begin{equation*}
\begin{split}
N(d_1 d_1', d_2 d_2') &\leq \|d_1\|_{\D_1} N(d_1', d_2') + \|d_2'\|_{\D_2} N(d_1,d_2)\\
&\leq \|(d_1,d_2)\|_{\D_1\oplus \D_2}N(d_1', d_2') + \|(d'_1,d'_2)\|_{\D_1\oplus\D_2}N(d_1,d_2)\\
&\leq \|(d_1,d_2)\|_{\D_1\oplus \D_2} \Lip(d_1', d_2') + \|(d_1',d_2')\|_{\D_1\oplus\D_2} \Lip(d_1,d_2)\\
&\leq F(\|(d_1,d_2)\|_{\D_1\oplus\D_2}, \|(d_1',d_2')\|_{\D_1\oplus\D_2}, \Lip(d_1,d_2), \Lip(d_1',d_2'))\\ &\quad \text{ by Assertion (\ref{permissible-eq}) of Definition (\ref{permissible-def}) \text{.}}
\end{split}
\end{equation*} 

We then deduce immediately that for all $d_1,d_1'\in\dom{\Lip_1}$ and $d_2,d_2'\in\dom{\Lip_2}$:
\begin{multline*}
\max\{ N\left(\Jordan{d_1}{d_1'},\Jordan{d_2}{d_2'}\right), N\left(\Lie{d_1}{d_1'},\Lie{d_2}{d_2'}\right) \}\\ \leq F(\|(d_1,d_2)\|_{\D_1\oplus\D_2}, \|(d_1',d_2')\|_{\D_1\oplus\D_2}, \Lip(d_1,d_2), \Lip(d_1',d_2'))\text{.}
\end{multline*}

Thus it follows immediately that $\Lip$ is an $F$-quasi-Leibniz seminorm since $\Lip_1$ and $\Lip_2$ are.

The proof of \cite[Theorem 3.1]{Latremoliere14} now applies to reach the conclusion of this theorem.
\end{proof}

We thus define, following \cite[Definition 3.5]{Latremoliere14}:
\begin{definition}\label{appropriate-def}
Let $F$ be a permissible function. Let $\mathcal{C}$ be a nonempty class of {\Qqcms{F}s}. A class $\mathcal{T}$ of $F$-tunnels is $(\mathcal{C},F)$-\emph{appropriate} when:
\begin{enumerate}
\item for all $(\A,\Lip_\A)$ and $(\B,\Lip_\B)$ in $\mathcal{C}$, there exists $\tau \in \mathcal{T}$ from $(\A,\Lip_\A)$ to $(\B,\Lip_\B)$,
\item if there exists an isometric isomorpism $h : (\A,\Lip_\A) \rightarrow (\B,\Lip_\B)$ for any two $(\A,\Lip_\A)$, $(\B,\Lip_\B)$ in $\mathcal{C}$, then both $(\A,\Lip_\A,\mathrm{id}_\A,h)$ and $(\B,\Lip_\B,h^{-1},\mathrm{id}_\B)$ are elements of $\mathcal{T}$, where $\mathrm{id}_E$ is the identity map on any set $E$,
\item if $\tau = (\D,\Lip,\pi,\rho)\in\mathcal{T}$ then $\tau^{-1} = (\D,\Lip,\rho,\pi)\in\mathcal{T}$,
\item if $(\D,\Lip,\pi,\rho)\in\mathcal{T}$ then the co-domains of $\pi$ and $\rho$ lie in $\mathcal{C}$,
\item if $\tau_1,\tau_2\in\mathcal{T}$ and $\varepsilon > 0$, then the tunnel $\tau_3$ defined by Theorem (\ref{tunnel-composition-thm}) also lies in $\mathcal{T}$.
\end{enumerate}
\end{definition}

Let $F$ be a permissible function, and let $\mathcal{C}$ be a nonempty class of {\Qqcms{F}s}. Let $\mathcal{T}$ be a $(\mathcal{C},F)$-appropriate class of $F$-tunnels. Let $(\A,\Lip_\A)$ and $(\B,\Lip_\B)$ be in $\mathcal{C}$. The set of all $F$-tunnels in $\mathcal{T}$ from $(\A,\Lip_\A)$ to $(\B,\Lip_\B)$ is denoted by:
\begin{equation*}
\tunnelset{\A,\Lip_\A}{\B,\Lip_\B}{\mathcal{T}}\text{.}
\end{equation*}

The main tool for our work is a distance constructed on the class of {\gQqcms s}. This distance is a form of our dual propinquity adapted to {\gQqcms s}, and its construction is modeled after the Gromov-Hausdorff distance \cite{Gromov,Gromov81}:

\begin{definition}\label{propinquity-def}
Let $F$ be a permissible function. Let $\mathcal{C}$ be some nonempty class of {\Qqcms{F}s} and let $\mathcal{T}$ be a class of $(\mathcal{C},F)$-appropriate tunnels. 

The \emph{$\mathcal{T}$-dual Gromov-Hausdorff propinquity} $\propinquity{\mathcal{T}}((\A,\Lip_\A),(\B,\Lip_\B))$ between two {\Qqcms{F}s} $(\A,\Lip_\A)$ and $(\B,\Lip_\B)$ is the number:
\begin{equation*}
\inf\left\{ \tunnelextent{\tau} : \tau \in \tunnelset{\A,\Lip_\A}{\B,\Lip_\B}{\mathcal{T}} \right\}\text{.}
\end{equation*}
\end{definition}

\begin{notation}
We simply shall write $\propinquity{F}$ for the dual propinquity defined on the class of all {\Qqcms{F}s}, using the class of all $F$-tunnels, for any permissible function $F$. Moreover, if $C\geq 1$ and $D\geq 0$ are given, and if $F_{C,D} : x,y,l_x,l_y \mapsto C(x l_y + y l_x) + D l_x l_y$, then $\propinquity{F_{C,D}}$ will simply be denoted by $\propinquity{C,D}$. 
\end{notation}

The original construction of the dual propinquity \cite{Latremoliere13c} involved the \emph{length} of a tunnel, rather than its extent. The extent \cite{Latremoliere14} was useful to obtain the triangle inequality more easily. However, estimates on the dual propinquity may be easier to derive using our notion of length. The length and extent are equivalent, in the sense described below.

\begin{definition}\label{length-def}
Let $F$ be a permissible function. Let $\tau = (\D,\Lip_\D,\pi_\A,\pi_\B)$ be a $F$-tunnel from $(\A,\Lip_\A)$ to $(\B,\Lip_\B)$.
\begin{enumerate}
\item The \emph{reach} of $\tau$ is:
\begin{equation*}
\tunnelreach{\tau} = \Haus{\Kantorovich{\Lip_\D}}\left(\pi_\A^\ast\left(\StateSpace(\A)\right),\pi_\B^\ast\left(\StateSpace(\B)\right)\right)\text{.}
\end{equation*}
\item The \emph{depth} of $\tau$ is:
\begin{equation*}
\tunneldepth{\tau} = \Haus{\Kantorovich{\Lip_\D}}\left(\StateSpace(\D), \co{\pi_\A^\ast\left(\StateSpace(\A)\right)\cup\pi_\B^\ast\left(\StateSpace(\B)\right)}\right)\text{,}
\end{equation*}
where $\co{X}$ is the closed convex hull of $X$.
\item The \emph{length} of $\tau$ is:
\begin{equation*}
\tunnellength{\tau} = \max\left\{ \tunneldepth{\tau}, \tunnelreach{\tau} \right\}\text{.}
\end{equation*}
\end{enumerate}
\end{definition}

Using the proof of \cite[Proposition 2.12]{Latremoliere14}, we obtain:

\begin{proposition}
Let $F$ be a permissible function. For any $F$-tunnel $\tau$, we have:
\begin{equation*}
\tunnellength{\tau}\leq\tunnelextent{\tau}\leq 2\tunnellength{\tau}\text{.}
\end{equation*}
\end{proposition}

In \cite{Latremoliere13c,Latremoliere14}, we proved that the dual propinquity is a complete metric on the class of {\Lqcms s}; in particular it satisfies the triangle inequality and distance zero between two {\Lqcms s} implies that they are isometrically isomorphic. The same holds in our present context. We simply explain in the proof of our theorem how to modify some estimates in \cite{Latremoliere13c,Latremoliere14} to obtain the desired result.

\begin{notation}
The diameter of any metric space $(X,\mathrm{m})$ is denoted by $\diam{X}{\mathrm{m}}$.
\end{notation}

\begin{theorem}\label{prop-thm}
Let $F$ be an admissible function. Let $\mathcal{C}$ be some nonempty class of {\Qqcms{F}s} and let $\mathcal{T}$ be a $(\mathcal{C},F)$-appropriate class of $F$-tunnels. 

The $\mathcal{T}$-dual propinquity $\propinquity{\mathcal{T}}$ is a metric on $\mathcal{C}$, i.e. for all $(\A,\Lip_\A)$, $(\B,\Lip_\B)$ and $(\D,\Lip_\D)$ in $\mathcal{C}$:
\begin{enumerate}
\item We have:
\begin{equation*}
\propinquity{\mathcal{T}}((\A,\Lip_\A),(\B,\Lip_\B)) < \infty\text{,}
\end{equation*}
\item $\propinquity{\mathcal{T}}((\A,\Lip_\A),(\B,\Lip_\B)) = \propinquity{\mathcal{T}}((\B,\Lip_\B),(\A,\Lip_\A))$,
\item We have:
\begin{equation*}
\propinquity{\mathcal{T}}((\A,\Lip_\A),(\D,\Lip_\D)) \leq \propinquity{\mathcal{T}}((\A,\Lip_\A),(\B,\Lip_\B)) + \propinquity{\mathcal{T}}((\B,\Lip_\B),(\D,\Lip_\D))\text{,}
\end{equation*}
\item $\propinquity{\mathcal{T}}((\A,\Lip_\A),(\B,\Lip_\B)) = 0$ if and only if there exists an isometric isomorphism between $(\A,\Lip_\A)$ and $(\B,\Lip_\B)$,
\item if $\dist_q$ is the quantum Gromov-Hausdorff distance, then:
\begin{equation*}
\dist_q((\A,\Lip_\A),(\B,\Lip_\B)) \leq \propinquity{\mathcal{T}}((\A,\Lip_\A),(\B,\Lip_\B))\text{.}
\end{equation*}
\end{enumerate}

Moreover for all {\Qqcms{F}s} $(\A,\Lip_\A)$ and $(\B,\Lip_\B)$, we have:
\begin{equation*}
\propinquity{F}((\A,\Lip_\A),(\B,\Lip_\B)) \leq \max\left\{\diam{\StateSpace(\A)}{\Kantorovich{\Lip_\A}}, \diam{\StateSpace(\B)}{\Kantorovich{\Lip_\B}}\right\}\text{.}
\end{equation*}

Furthermore, if $\mathrm{GH}$ is the Gromov-Hausdorff distance, $(X,\mathrm{d}_X)$ and $(Y,\mathrm{d}_Y)$ are two compact metric spaces, and if $\mathrm{Lip}_X$ and $\mathrm{Lip}_Y$ are the Lipschitz seminorms associated, respectively, with the metrics $\mathrm{d}_X$ and $\mathrm{d_Y}$, then:
\begin{equation*}
\propinquity{F}\left((C(X),\mathrm{Lip}_X),(C(Y),\mathrm{Lip}_Y)\right)\leq \mathrm{GH}((X,\mathrm{d}_X),(Y,\mathrm{d}_Y)) \text{,}
\end{equation*}
where $C(X)$ and $C(Y)$ are the respective C*-algebras of $\C$-valued continuous functions over $X$ and $Y$.

Last, if $F$ is continuous and permissible, then $\propinquity{F}$ is a complete metric on the class of all {\Qqcms{F}s}.

\end{theorem}

\begin{proof}
The proofs are given in \cite{Latremoliere13b,Latremoliere14}, for the case when $F : x,y,l_x,l_y \mapsto x l_y + y l_x$. The more general setting of this paper is however an immediate consequence of the work in \cite{Latremoliere13b,Latremoliere14}, once we make the following simple observations.
\begin{enumerate}
\item The dual propinquity is finite by Definition (\ref{appropriate-def}). Moreover, Proposition (\ref{tunnels-exist-prop}) shows that, if we simply work with all $F$-tunnels, we obtain the desired upper bound for the propinquity between two {\Qqcms{F}s}.
\item The dual propinquity is symmetric by Definition (\ref{appropriate-def}) and since the extent of a tunnel equals to the extent of its inverse.
\item Using Theorem (\ref{tunnel-composition-thm}) in place of \cite[Theorem 3.1]{Latremoliere14}, we can deduce that the dual propinquity satisfies the triangle inequality on the class of {\Qqcms{F}s} as in \cite[Theorem 3.7]{Latremoliere14}.
\item \cite[Lemma 4.2, Proposition 4.4, Corollary 4.5, Proposition 4.6]{Latremoliere13b} remain unchanged in our setting.
\item \cite[Proposition 2.12]{Latremoliere14} remains unchanged as well.
\item The estimate in \cite[Proposition 4.8]{Latremoliere13b} should be modified very slightly. Let $\tau = (\D,\Lip_\D,\pi_\A,\pi_\B)$ be an $F$-tunnel from $(\A,\Lip_\A)$ to $(\B,\Lip_\B)$. We recall from \cite[Definition 4.1]{Latremoliere13b} that for $a\in\sa{\A}$ with $\Lip_\A(a)<\infty$ and for all $l \geq \Lip_\A(a)$, the \emph{target set} $\targetsettunnel{\tau}{a}{l}$ is the set:
\begin{equation*}
\targetsettunnel{\tau}{a}{l} = \left\{ \pi_\B(d) : d \in \sa{\D}, \pi_\A(d) = a, \Lip_\D(d) \leq l \right\}\text{.}
\end{equation*}

If $a,a' \in \sa{\A}$ with $\max\{\Lip_\A(a),\Lip_\A(a')\} \leq l$ for some $l > 0$, and if $d,d'\in\sa{\D}$ with $\pi_\A(d)=a$, $\pi_\A(d') = a'$ and $\max\{\Lip_\D(d),\Lip_\D(d')\} \leq l$, then by \cite[Proposition 4.6]{Latremoliere13b}:
\begin{equation*}
\|d\|_\D \leq \|a\|_\A + 2 l \tunnelextent{\tau} \text{ and }\|d'\|_\D \leq \|a'\|_\A + 2 l \tunnelextent{\tau}\text{.}
\end{equation*}
Consequently, using the $F$-quasi-Leibniz property of $\Lip_\D$:
\begin{equation}\label{prop-thm-eq1}
\begin{split}
\Lip_\D\left(\Jordan{d}{d'}\right) &\leq F(\|d\|_\D, \|d'\|_\D, \Lip_\D(d), \Lip_\D(d'))\\
&\leq F(\|a\|_\A+ 2 l \tunnelextent{\tau}, \|a'\|_\A + 2 l \tunnelextent{\tau}, l, l) {.}
\end{split}
\end{equation}

Thus one may conclude that, if $b \in \targetsettunnel{\tau}{a}{l}$ and $b'\in\targetsettunnel{\tau}{a'}{l}$ then:
\begin{equation}\label{prop-thm-eq2}
\Jordan{b}{b'} \in \targetsettunnel{\tau}{\Jordan{a}{a'}}{F(\|a\|_\A+2 l \tunnelextent{\tau}, \|a'\|_\A + 2 l \tunnelextent{\tau}, l, l)}\text{.}
\end{equation}
A similar modification of \cite[Proposition 4.8]{Latremoliere13b} holds for the Lie product. Moreover, we observe that we could substitute the tunnel length to its extent in the above inequalities.

\item One may now check that the proof of \cite[Theorem 4.16]{Latremoliere13b} carries unchanged, except with the use of Expression (\ref{prop-thm-eq2}) in lieu of \cite[Corollary 4.14, Proposition 4.8]{Latremoliere13b}. It is easy to verify that the change in the constant involved in the target set does not impact the proof at all: the key is that this value does not depend on the choices of lifts. Thus \cite[Theorem 4.16]{Latremoliere13b} holds.

\item \cite[Theorem 5.5, Corollary 5.6, Corollary 5.7]{Latremoliere13c} all carry over unchanged thanks to Assertion (\ref{permissible-eq}) of Definition (\ref{permissible-def}). We will make a brief comment after this theorem about the quantum propinquity and {\Qqcms{F}s}.

\item For completeness, we assume that $F$ is \emph{continuous}, in addition to being permissible. \cite[Section 6]{Latremoliere13b} can be entirely rewritten with $F$-quasi-Lip-norms. The only place where we must be careful is in the proof of \cite[Lemma 6.24]{Latremoliere13c}. Using the notations of \cite[Hypothesis 6.4, 6.5]{Latremoliere13c}, we note that the seminorms $\mathrm{S}_N$ in \cite[Lemma 6.24]{Latremoliere13c} is $F$-quasi-Leibniz, and thus, with the seminorm $\mathrm{Q}$ of \cite[Lemma 6.24]{Latremoliere13c} satisfies, for all $f, g \in \alg{F}$ and for all $\varepsilon > 0$:
\begin{equation*}
\mathrm{Q}(\Jordan{f}{g}) \leq F\left(\|f\|_{\alg{F}}+\varepsilon, \|g\|_{\alg{F}}+\varepsilon, \mathrm{Q}(f)+\varepsilon, \mathrm{Q}(g)+\varepsilon \right)\text{.}
\end{equation*}
Using the continuity of $F$, we thus get:
\begin{equation*}
\mathrm{Q}(\Jordan{f}{g}) \leq F\left(\|f\|_{\alg{F}}, \|g\|_{\alg{F}}, \mathrm{Q}(f), \mathrm{Q}(g) \right)\text{.}
\end{equation*}
The same process applies to the Lie product. Thus $\mathrm{Q}$ is indeed a {\Qqcms{F}}. Thus, a Cauchy sequence of {\Qqcms{F}s} converges, for the dual propinquity, to a {\Qqcms{F}}.
\end{enumerate}

This concludes our proof.
\end{proof}

\begin{remark}
We emphasize that Theorem (\ref{prop-thm}) requires that we first fix a permissible function $F$; if we were working on the class of all {\gQqcms s}, then the estimates needed for Theorem (\ref{prop-thm}) would not be valid any longer, since we would essentially lose control of the quasi-Leibniz property.
\end{remark}

We conclude this section with a brief remark on extending the quantum propinquity to {\gQqcms s}. The quantum propinquity \cite{Latremoliere13} is a specialized version of the dual propinquity, where for any two {\Lqcms s} $(\A,\Lip_\A)$, $(\B,\Lip_\B)$, all the tunnels from $(\A,\Lip_\A)$ to $(\B,\Lip_\B)$ are of the form:
\begin{equation*}
(\A\oplus\B, \Lip, \pi_\A,\pi_\B)
\end{equation*}
where $\pi_\A$ and $\pi_\B$ are the canonical surjections onto $\A$ and $\B$, respectively, and for all $(a,b)\in\A\oplus\B$, the Lip-norm $\Lip$ is of the form:
\begin{equation}\label{qp-lip-eq}
\Lip(a,b) = \max\left\{\Lip_\A(a),\Lip_\B(b),\frac{1}{\lambda}\|\rho_\A(a)\omega - \omega\rho_\B(b)\|_\D \right\}
\end{equation}
where $\D$ is some C*-algebra, $\rho_\A:\A\hookrightarrow\D$ and $\rho_\B:\B\hookrightarrow \D$ are unital *-monomorphisms, and  $\omega\in\D$ is an element of $\D$ such that there exists at least one state $\varphi \in\StateSpace(\D)$ such that:
\begin{equation*}
\varphi((\unit_\D-\omega)(\unit_\D-\omega)^\ast) = \varphi((\unit_\D-\omega)^\ast(\unit_\D-\omega)) = 0\text{.}
\end{equation*}
Of course, the key question is to choose $\lambda > 0$ appropriately in Expression (\ref{qp-lip-eq}). This intriguing question is in fact the entire purpose of \cite{Latremoliere13}, which presents the construction of the quantum propinquity in quite a different manner.

The quadruples $\gamma = (\D,\omega,\rho_\A,\rho_\B)$ are called bridges in \cite{Latremoliere13}. The quantum propinquity was our first construction of a metric adapted to {\Lqcms s}, and the particular form of tunnels involved provide some potential algebraic advantages, such as the construction of matricial versions of these Lip-norms. As any specialization, the quantum propinquity is stronger than the dual propinquity in general. The key aspect of the construction of the quantum propinquity is to associate a length $\bridgelength{\gamma}{\Lip_\A,\Lip_\B}$ to a bridge $\gamma$, and then $\lambda$ in Expression (\ref{qp-lip-eq}) may be chosen as $\bridgelength{\gamma}{\Lip_\A,\Lip_\B} + \varepsilon$ for any $\varepsilon > 0$. With such a choice, the length of the tunnel given by Expression (\ref{qp-lip-eq}) is no more than $\lambda$. In particular, the quantum propinquity provides a useful technique to compute bounds for the dual propinquity (see for instance \cite{Latremoliere13b}).

Assertion (\ref{permissible-eq}) of Definition (\ref{permissible-def}) proves that Expression (\ref{qp-lip-eq}) leads to a $F$-quasi-Lip-norm if $\Lip_\A$, $\Lip_\B$ are $F$-quasi-Lip-norms. One can then verify that the quantum propinquity extends to a metric to the class of {\Qqcms{F}s} for any permissible $F$, without any modification to our notion of bridge.

\section{Compactness for some classes of Finite Dimensional {\gQqcms s}}

We begin this section with a characterization of quasi-Leibniz seminorms based upon their unit balls. We recall that a convex subset $S\subseteq \A$ is \emph{balanced} if $\{\lambda x : x\in S\} \subseteq S$ for all $\lambda\in[-1,1]$. If $\Lip$ is a seminorm on some $\R$-vector space $A$ then $\{x\in A : \Lip(x)\leq 1\}$ is convex and balanced. Conversely, if $S$ is a convex and balanced subset of $A$, then the \emph{gauge seminorm} or \emph{Minkowsky functional} $\Lip$ associated with $S$ is the seminorm on $A$ defined by $\Lip(x) = \inf\{\lambda > 0 : x \in \lambda S \}$ for all $x\in A$. In particular, if $A$ is a normed vector space, then these constructions establish a bijection between the set of closed, convex balanced subsets of $A$ and the seminorms on $A$ which are lower semi-continuous with respect to the norm of $A$. 

\begin{lemma}\label{quasi-Leibniz-lemma}
Let $\A$ be a C*-algebra. Let $\alg{S}$ be a closed, balanced convex subset of $\sa{\A}$. Let $F:[0,\infty)^4\rightarrow[0,\infty)$ such that:
\begin{enumerate}
\item $F$ is a permissible function,
\item for all $(x,y,l_x,l_y)\in [0,\infty)^4$ and for all $\lambda,\mu\geq 0$, we have:
\begin{equation*}
F(\lambda x, \mu y, \lambda l_x, \mu l_y) \geq \lambda\mu F(x,y,l_x,l_y)\text{.}
\end{equation*}
\item $F(x,y,\cdot,\cdot)$ is continuous for all $x,y \in [0,\infty)$. 
\end{enumerate}

The Minkowsky functional $\Lip$ of $\alg{S}$ is $F$-quasi-Leibniz if and only if for all $a,b \in \alg{S}$:
\begin{equation*}
\Jordan{a}{b} \in F(\|a\|_\A,\|b\|_\A,1,1) \alg{S}
\end{equation*}
and
\begin{equation*}
\Lie{a}{b} \in F(\|a\|_\A,\|b\|_\A,1,1)\alg{S}\text{.}
\end{equation*}
\end{lemma}

\begin{proof}
Assume that $\Lip$ is a $F$-quasi-Leibniz seminorm. Let $a,b \in \alg{S}$. If $a=0$ or $b=0$ then $\Jordan{a}{b} = \Lie{a}{b} = 0$ and thus:
\begin{equation*}
\Jordan{a}{b},\Lie{a}{b} \in 0\cdot \alg{S} \subseteq F(\|a\|_\A,\|b\|_\A,1,1)\alg{S}
\end{equation*}
since $0\in\alg{S}$. Henceforth, assume $a, b \not = 0$.

We observe, as $F$ is a permissible function, $\Lip(a),\Lip(b)\leq 1$ and $\Lip$ is $F$-quasi-Leibniz:
\begin{equation*}
\Lip(\Jordan{a}{b}) \leq F(\|a\|_\A,\|b\|_\A,\Lip(a),\Lip(b))\leq F(\|a\|_\A,\|b\|_\A,1,1)\text{.}
\end{equation*}
Since $a$ and $b$ are nonzero, $F(\|a\|_\A,\|b\|_\A,1,1) \geq \|a\|_\A + \|b\|_\A > 0$, and thus $\Lip\left(\frac{1}{F(a,b,1,1)}\Jordan{a}{b}\right) \leq 1$, i.e. $\frac{1}{F(\|a\|_\A,\|b\|_\A,1,1)}\Jordan{a}{b} \in \alg{S}$, so:
\begin{equation*}
\Jordan{a}{b} \in F\left(\|a\|_\A,\|b\|_\A,1,1\right)\alg{S}\text{.}
\end{equation*}

We obtain a similar result for the Lie product, and thus our condition is necessary.

Assume conversely, that for all $a,b\in\alg{S}$ we have:
\begin{equation*}
\Jordan{a}{b}, \Lie{a}{b} \in F\left(\|a\|_\A,\|b\|_\A,1,1\right)\alg{S}\text{.}
\end{equation*}

Let $a,b \in \dom{\Lip}$. Assume first that $\Lip(a)\Lip(b) > 0$. We then have $\Lip(a)^{-1}a, \Lip(b)^{-1}b \in \alg{S}$ and thus:
\begin{equation*}
\frac{1}{\Lip(a)\Lip(b)}\Jordan{a}{b} = \Jordan{\frac{1}{\Lip(a)}a}{\frac{1}{\Lip(b)}b} \in F\left(\frac{\|a\|_\A}{\Lip(a)},\frac{\|b\|_\A}{\Lip(b)},1,1\right) \alg{S}
\end{equation*}
i.e. 
\begin{equation*}
\Lip\left(\frac{1}{\Lip(a)\Lip(b)}\Jordan{a}{b} \right) \leq F\left(\frac{\|a\|_\A}{\Lip(a)},\frac{\|b\|_\A}{\Lip(b)},1,1\right)\text{.}
\end{equation*}

Consequently:
\begin{equation*}
\begin{split}
\Lip(\Jordan{a}{b}) &= \Lip(a)\Lip(b)\Lip\left(\frac{1}{\Lip(a)\Lip(b)}\Jordan{a}{b}\right)\\
&\leq \Lip(a)\Lip(b) F\left(\frac{\|a\|_\A}{\Lip(a)}, \frac{\|b\|_\A}{\Lip(b)}, 1,1\right)\\
&\leq F\left(\|a\|_\A, \|b\|_\A, \Lip(a), \Lip(b)\right) \text{.}\\
\end{split}
\end{equation*}

Now, assume $\Lip(a) = 0$ and $\Lip(b) > 0$. Let $\lambda > 0$. Since $\Lip(\lambda^{-1}a) = 0$ we have $\lambda^{-1}a \in \alg{S}$. Of course, $\Lip(b)^{-1}b \in \alg{S}$. Hence:
\begin{align*}
\Jordan{\left(\lambda^{-1} a\right)}{ \left(\Lip^{-1}(b) b\right)} &\in F\left(\lambda^{-1}\|a\|_\A,\frac{\|b\|_\A}{\Lip(b)},1,1\right)\alg{S} \\ \intertext{so}
\Lip\left(\Jordan{\left(\lambda^{-1} a\right)}{ \left(\Lip^{-1}(b) b\right) }\right) &\leq F\left(\lambda^{-1}\|a\|_\A,\frac{\|b\|_\A}{\Lip(b)},1,1\right)
\end{align*}
and thus:
\begin{equation*}
\Lip(\Jordan{a}{b}) \leq F\left(\|a\|_\A,\|b\|_\B, \lambda, \Lip(b)\right) \text{.}
\end{equation*}

As $\lambda > 0$ is arbitrary, by partial continuity of $F$, we conclude that:
\begin{equation*}
\Lip\left(\Jordan{a}{b}\right) \leq F(\|a\|_\A,\|b\|_\A,0,\Lip(b)) = F(\|a\|_\A,\|b\|_\A,\Lip(a),\Lip(b))\text{,}
\end{equation*}
as desired. The case $\Lip(a)>0$, $\Lip(b) = 0$ is dealt with symmetrically. 

Last, if $\Lip(a)=\Lip(b) = 0$, then again, for all $\mu, \lambda > 0$, we have:
\begin{equation*}
\Jordan{\left(\lambda^{-1} a\right)}{\left( \mu^{-1} b\right)} \in F\left(\lambda^{-1}\|a\|_\A,\mu^{-1}\|b\|_\B,1,1\right)\alg{S}
\end{equation*}
and thus $\Lip(\Jordan{a}{b}) \leq F(\|a\|_\A,\|b\|_\A,\lambda,\mu)$. Again, since $\lambda, \mu$ are arbitrary positive numbers and by partial continuity of $F$, we conclude that:
\begin{equation*}
\Lip(\Jordan{a}{b}) \leq F(\|a\|_\A,\|b\|_\A,0,0) = F(\|a\|_\A,\|b\|_\A,\Lip(a),\Lip(b)) \text{.}
\end{equation*}

The proof is analogous for the Lie product. This concludes our lemma.
\end{proof}

\begin{corollary}\label{quasi-Leibniz-corollary}
Let $\A$ be a C*-algebra. Let $\alg{S}$ be a closed, balanced convex subset of $\sa{\A}$. Let $C\geq 1$ and $D\geq 0$. The Minkowsky functional $\Lip$ of $\alg{S}$ is $(C,D)$-quasi-Leibniz if and only if for all $a,b \in \alg{S}$:
\begin{equation*}
\Jordan{a}{b} \in \left[ C\left(\|a\|_\A+\|b\|_\A\right) + D\right]\alg{S}
\end{equation*}
and
\begin{equation*}
\Lie{a}{b} \in \left[ C\left(\|a\|_\A+\|b\|_\A\right) + D\right]\alg{S}\text{.}
\end{equation*}
\end{corollary}

\begin{proof}
Apply Lemma (\ref{quasi-Leibniz-lemma}) with $F: (x,y,l_x,l_y) \in [0,\infty)^4 \mapsto C(x l_y + y l_x) + D l_x l_y$.
\end{proof}

Corollary (\ref{quasi-Leibniz-corollary}) demonstrates the application of our work to a special case of {\gQqcms s} which will play a central role in the last section of this paper. Other possible examples of applications of Lemma (\ref{quasi-Leibniz-lemma}) include such functions as $(x,y,l_x,l_y)\mapsto C\left(x^p l_y^p + y^p l_x^p + D l_x^p l_y^p \right)^{\frac{1}{p}}$ for $p\geq 1$, $D\geq 0$ and $C\geq 2^{1-\frac{1}{p}}$, and the sum of such functions with the permissible functions of Corollary (\ref{quasi-Leibniz-corollary}), for instance.

One important use for Lemma (\ref{quasi-Leibniz-lemma}) is to help determine whether a seminorm whose unit ball is constructed as the Hausdorff limit of the unit balls of quasi-Leibniz Lip-norms possesses a quasi-Leibniz property. The following result will help with these considerations.

\begin{lemma}\label{Haus-limit-lemma}
Let $\A$ be a C*-algebra, and let $F:[0,\infty)^4\mapsto [0,\infty)$ be a permissible function such that $F(\cdot,\cdot,1,1)$ is continuous. For all $n\in\N$, let $\alg{L}_n$ be a closed, balanced convex subset of $\sa{\A}$ such that, for all $a,b \in \alg{L}_n$ we have:
\begin{equation*}
\Jordan{a}{a'}, \Lie{a}{a'} \in F\left(\|a\|_\A,\|b\|_\A,1,1\right)\alg{L}_n\text{.}
\end{equation*}
If $\left(\alg{L}_n\right)_{n\in\N}$ converges to some closed set $\alg{L}$ for the Hausdorff distance associated with the norm $\|\cdot\|_\A$, then $\alg{L}$ is a balanced convex subset of $\sa{\A}$ such that, for all $a,a'\in \alg{L}$, we also have:
\begin{equation*}
\Jordan{a}{a'}, \Lie{a}{a'} \in F\left(\|a\|_\A, \|b\|_\A,1,1\right)\alg{L}\text{.}
\end{equation*}
\end{lemma}

\begin{proof}
Let $\Haus{}$ be the Hausdorff distance on closed subsets of $\A$, associated with $\|\cdot\|_\A$.

Let $a,b \in \alg{L}$ and let $\varepsilon > 0$. If either $a$ or $b$ is $0$ then $ab = 0 \in F(\|a\|_\A,\|b\|_\A,1,1)\alg{L}$ trivially, so we henceforth assume that $a\not=0, b\not=0$ and we further impose that $\varepsilon < \frac{1}{2}\min\{\|a\|_\A,\|b\|_\A\}$.

By continuity of $F(\cdot,\cdot,1,1)$, there exists $\delta > 0$ such that for all $x,y \geq 0$ with $|x-\|a\|_\A|\leq\delta$, $|y-\|b\|_\A|\leq\delta$, we have:
\begin{equation*}
|F(\|a\|_\A,\|b\|_\A,1,1) - F(x,y,1,1)| \leq\varepsilon\text{.}
\end{equation*}

Let $\alpha = \min \{\varepsilon,\delta\}$.

There exists $N\in\N$ such that for all $n\geq N$ we have $\Haus{}(\alg{L}_{n},\alg{L})\leq \alpha$. Let $n\geq N$. 

Thus there exists $c_n,d_n \in \alg{L}_{n}$ such that:
\begin{equation*}
\|a-c_n\|_\A \leq\alpha\text{ and }\|b-d_n\|_\A \leq\alpha\text{.}
\end{equation*}

Note that thanks to our choice of $\varepsilon$, we can assert that $c_n, d_n \not= 0$.

By assumption on $\Lip_{n}$, we have:
\begin{equation*}
\Jordan{c_n}{d_n} \in F\left(\|c_n\|_\A,\|d_n\|_\A,1,1\right) \alg{L}_{n}\text{.}
\end{equation*}

To ease notations, let $\lambda_n = F(\|c_n\|_\A,\|d_n\|_\A,1,1)$. Note that since $|\|c_n\|_\A-\|a\|_\A|\leq\delta$ and $|\|d_n\|-\|b\|_\A|\leq\delta$, we have:
\begin{equation*}
|\lambda_n - F(\|a\|_\A,\|b\|_\A,1,1)| \leq\varepsilon\text{.}
\end{equation*}

Let $e_n \in \alg{L}$ such that $\|\Jordan{c_n}{d_n} - \lambda_n e_n\|_\A \leq \lambda_n \alpha$.

On the other hand, we note that since $a,b,c_n,d_n$ are self-adjoint:
\begin{align*}
\|\Jordan{a}{b} - \Jordan{c_n}{d_n}\|_\A &\leq \frac{1}{2}\left( \|ab-c_n d_n\|_\A + \|ba-d_n c_n\|_\A  \right)\\
&= \frac{1}{2}\left(\|ab-c_n d_n\|_\A + \|(ab-c_n d_n)^\ast\|_\A\right) = \|ab-c_n d_n\|_\A\\
&\leq \|a\|_\A\|b-d_n\|_\A + \|d_n\|_\A\|a-c_n\|_\A\text{.}
\end{align*}

Therefore:
\begin{equation*}
\|\Jordan{a}{b}-\Jordan{c_n}{d_n}\|_\A \leq \|a\|_\A\|b-d_n\|_\A + \|d_n\|_\A \|a-c_n\|_\A \leq \|a\|_\A \alpha + (\|b\|_\A +\alpha)\alpha\text{.}
\end{equation*}

We note that $F(x,y,1,1)\geq x+y > 0$ if $x$ or $y$ are strictly positive, so:
\begin{equation*}
\begin{split}
\|e_n\|_\A &\leq \alpha + \lambda_n^{-1} \|\Jordan{c_n}{d_n}\|_\A \\
&\leq \alpha + \frac{\|c_n\|_\A\|d_n\|_\A}{F(\|c_n\|_\A,\|d_n\|_\A,1,1)}  \\
&\leq  \alpha + \frac{\left(\|a\|_\A+\varepsilon\right)\left(\|b\|_\A+\varepsilon\right)}{F(\|a\|_\A-\varepsilon, \|b\|_\A-\varepsilon, 1,1)}\\
&\leq \alpha + \frac{9 \|a\|_\A\|b\|_\A}{4 F\left(\frac{1}{2}\|a\|_\A, \frac{1}{2}\|b\|_\A,1,1\right)} \text{.}
\end{split}
\end{equation*}
Consequently:
\begin{equation*}
\begin{split}
\|\Jordan{a}{b} &- F(\|a\|_\A,\|b\|_\A,1,1)e_n\|_\A \\
&\leq \|\Jordan{a}{b} - \Jordan{c_n}{d_n}\|_\A + \|\Jordan{c_n}{d_n} - \lambda_n e_n\|_\A \\
&\quad + \left\|\left(F(\|a\|_\A,\|b\|_\A,1,1) - \lambda_n  \right)e_n\right\|_\A\\
&\leq \varepsilon(\|a\|_\A+\|b\|_\A+\varepsilon) +\lambda_n \varepsilon \\
&\quad + \left|F(\|a\|_\A,\|b\|_\A,1,1)-\lambda_n \right| \|e_n\|_\A\\
& \leq \varepsilon(\|a\|_\A+\|b\|_\A+\varepsilon) + (F(\|a\|_\A,\|b\|_\A,1,1)+\varepsilon)\varepsilon \\
&\quad + \varepsilon\|e_n\|_\A\\
& \leq \varepsilon(\|a\|_\A+\|b\|_\A+\varepsilon) + (F(\|a\|_\A,\|b\|_\A,1,1)+\varepsilon)\varepsilon \\
&\quad + \varepsilon\left(\varepsilon + \frac{9 \|a\|_\A\|b\|_\A}{4 F\left(\frac{1}{2}\|a\|_\A, \frac{1}{2}\|b\|_\A,1,1\right)} \right) \\
&= \mathscr{O}(\varepsilon)\text{.}
\end{split}
\end{equation*}
Hence, as $\varepsilon \in \left(0,\frac{1}{2}\min\{\|a\|_\A,\|b\|_\A\}\right)$ is arbitrary, by continuity of $F(\cdot,\cdot,1,1)$, and since $\alg{L}$ is closed, we conclude that:
\begin{equation*}
\Jordan{a}{b} \in F(\|a\|_\A,\|b\|_\A,1,1) \alg{L}\text{.}
\end{equation*}

The same proof holds for the Lie product. This proves the key fact of our lemma.

As the Hausdorff limit of a convex balanced set, it is easy to see that $\alg{L}$ is balanced and convex. Indeed, if $a, b\in\alg{L}$, then for all $n\in\N$ there exists $a_n, b_n \in \alg{L}_n$ such that $\|a-a_n\|_\A \leq \Haus{}(\alg{L}_n,\alg{L})+\frac{1}{n+1}$ and $\|b-b_n\|_\A\leq \Haus{}(\alg{L}_n,\alg{L})+\frac{1}{n+1}$.

Now, if $t \in [0,1]$, then $ta_n + (1-t)b_n \in\alg{L}_n$ since $\alg{L}_n$ is convex. Therefore, there exists $c_n \in \alg{L}$ such that $\|c_n - (ta_n+(1-t)b_n)\|_\A\leq \Haus{}(\alg{L}_n,\alg{L}) + \frac{1}{n+1}$. Now:
\begin{multline*}
\|ta+(1-t)b - c_n\|_\A \leq \|t(a-a_n) + (1-t)(b-b_n)\|_\A \\ + \|ta_n + (1-t)b_n - c_n\|_\A \leq 2\Haus{}(\alg{L}_n,\alg{L}) + \frac{2}{n+1} \text{.}
\end{multline*}

Since $\lim_{n\rightarrow\infty} 2\Haus{}(\alg{L},\alg{L}_n)+\frac{2}{n+1} = 0$, we conclude that $(c_n)_{n\in\N}$ converges to $ta+(1-t)b$, and since $\alg{L}$ is closed, we have $ta+(1-t)b \in \alg{L}$. The same reasoning applies to show that $\alg{L}$ is balanced as well.
\end{proof}

Lemmas (\ref{quasi-Leibniz-lemma}) and (\ref{Haus-limit-lemma}) suggest the following definition, which we will use in formulating our main results:
\begin{definition}\label{strongly-permissible-def}
A permissible function $F: [0,\infty)^4\rightarrow[0,\infty)$ is \emph{strongly permissible} when:
\begin{enumerate}
\item $F$ is continuous,
\item for all $(x,y,l_x,l_y)\in [0,\infty)^4$ and for all $\lambda,\mu \geq 0$ we have:
\begin{equation*}
\lambda \mu F(x,y,l_x,l_y) \leq F(\lambda x, \mu y, \lambda l_x, \mu l_y)\text{.}
\end{equation*}
\end{enumerate}
\end{definition}

We now establish a first compactness result, which serves as the basis for our main compactness theorem on {\gQqcms s} in the next section.

\begin{notation}
If $(\A,\Lip)$ is a {\gQqcms}, then the diameter of $\StateSpace(\A)$ for the {\mongekant} $\Kantorovich{\Lip}$ is denoted by $\diam{\A}{\Lip}$. 
\end{notation}

\begin{theorem}\label{f-compact-thm}
Let $F$ be a strongly permissible function and let $\mathcal{QQCMS}_F$ be the class of all {\Qqcms{F}s}. Let $d\in \N\setminus\{0\}$ and $K > 0$. The class:
\begin{equation*}
\mathcal{C}_{F,K,d} = \left\{ (\A,\Lip) \in \mathcal{QQCMS}_{F} : \dim_\C \A \leq d \text{ and }\diam{\A}{\Lip} \leq K \right\}
\end{equation*}
is compact for the dual propinquity $\propinquity{F}$.
\end{theorem}

\begin{proof}
Let $(\A_n,\Lip_n)_{n\in\N}$ be a sequence in $\mathcal{C}_{F,K,d}$. Let $\alg{M}_{d}$ be the C*-algebra of $d\times d$ matrices.

First, fix $n\in\N$.  We identify the C*-algebra $\A_n$ with a C*-subalgebra of $\alg{M}_d$ as follows. Up to $\ast$-isomorphism, we can write $\A_n = \oplus_{j \in J}\alg{M}_{t(j)}$ where $J = \{1,\ldots, d\}$ and $t(1)\geq t(2) \geq \cdots \geq t(d)$. We note that $t$ may be zero for some $j\in\{1,\ldots,d\}$, and the zero set of $t$ is a tail of $J$. 

Let $j\in \{1,\ldots,d\}$. Let $s(j) = \sum_{k=1}^j t(j)$ and set $s(0) = 0$. We now let $Q_j$ be the projection given as the diagonal matrix whose only nonzero entries are $1$ on the diagonal, from row $s(j-1)+1$ to $s(j)$, i.e. in block form:
\begin{equation*}
Q_j = \begin{pmatrix}
0_{s(j-1)} & & \\ & 1_{s(j)-s(j-1)} & \\ & & 0_{d - s(j)}
\end{pmatrix}\text{.}
\end{equation*}
 Of course the projections $Q_j$ are orthogonal and sum to the identity of $\alg{M}_d$. 

It then follows trivially that $\A_n$ is isomorphic to $\A'_n = \sum_{j\in J} Q_j\alg{M}_d Q_j$. Let $\pi_n : \A_n \rightarrow \A'_n$ be the *-isomorpism thus constructed. Note that $\pi_n$ is \emph{not} a unital map from $\A_n$ into $\alg{M}_d$.

If $\unit_n$ is the unit of $\A_n$ for all $n\in\N$, then $(\pi_n(\unit_n))_{n\in\N}$ is a sequence of diagonal projections, i.e. diagonal $d\times d$-matrices with entries in $\{0,1\}$. Thus, there exists a constant subsequence $(\pi_{g(n)}(\unit_{g(n)}))_{n\in\N}$, with value denoted by $p$, of $(\pi_n(\unit_n))_{n\in\N}$.

Let $\B = p \alg{M}_d p$. Note that for all $n\in\N$, the map $p\pi_{g(n)}p : \A_n \rightarrow \B$ is now a unital *-monomorphism. We shall henceforth omit the notation $\mathrm{ad}_p\pi_{g(n)}$ and simply identify $\A_{g(n)}$ with $p\pi_{g(n)}(\A_{g(n)})p$. We emphasize that with this identification, $\unit_{g(n)} = \unit_\B$.

Let $\alg{R} = \{ b \in \sa{\B} : \|b\|\leq K \}$ be the closed ball of center $0$ and radius $K$ in $\sa{\B}$. Since $\sa{\B}$ is finite dimensional, the set $\alg{R}$ is compact in norm. We shall denote by $\Haus{}$ the Hausdorff distance defined by the norm of $\sa{\B}$ on the compact subsets of $\alg{R}$. Since $\alg{R}$ is compact in norm, $\Haus{}$ induces a compact topology on the set of compact subsets of $\alg{R}$ as well \cite[Theorem 7.3.8]{burago01}.

Let us fix a state $\varphi$ of $\B$, and identify $\varphi$ with its restriction to $\A_{g(n)}$, which is a state of $\A_{g(n)}$, for all $n\in\N$. 

Now, for all $n\in\N$, let:
\begin{equation*}
\alg{L}_n = \left\{ a \in \sa{\A_{g(n)}} : \Lip_{g(n)}(a)\leq 1 \right \}
\end{equation*}
and
\begin{equation*}
\alg{D}_n = \left\{ a \in \alg{L}_n : \varphi(a) = 0 \right\}\text{.}
\end{equation*}
Fix $n\in\N$. By construction, we check that $\alg{L}_n = \alg{D}_n + \R\unit_\B$, since for all $a\in \sa{\A_{g(n)}}$ we check easily that $\Lip_{g(n)}(a \pm \varphi(a)\unit_n) = \Lip_{g(n)}(a)$. On the other other hand, we note that $\alg{D}_n$ is a compact subset of $\alg{R}$ since $\diam{\StateSpace(\A_n)}{\Kantorovich{\Lip_n}}\leq K$. Indeed, if $a\in\alg{D}_n$ then, for all $\psi\in\StateSpace(\A_{g(n)})$, we have:
\begin{equation*}
|\psi(a)| = |\psi(a)-\varphi(a)| \leq \Kantorovich{\Lip}(\varphi,\psi) \leq K\text{.}
\end{equation*}
Moreover, compactness of $\alg{D}_n$ follows from Theorem (\ref{Rieffel-thm}) since $(\A_{g(n)},\Lip_{g(n)})$ is a {\gQqcms} for all $n\in\N$.

 Thus, there exists a convergent subsequence $(\alg{D}_{f(n)})_{n\in\N}$ of $(\alg{D}_n)_{n\in\N}$ for $\Haus{}$, whose limit we denote by $\alg{D}$.

We now define $\alg{L} = \alg{D} + \R\unit_\B$. Let us first check that $(\alg{L}_{f(n)})_{n\in\N}$ converges to $\alg{L}$. Let $\varepsilon > 0$. There exists $N\in\N$ such that for all $n\geq N$, we have:
\begin{equation*}
\Haus{\|\cdot\|_\B}(\alg{D}_{f(n)},\alg{D})\leq\varepsilon\text{.}
\end{equation*}

Let $n\geq N$. We observe that, for any $a\in\alg{L}_{f(n)}$, there exists $a' \in \alg{D}_{f(n)}$ and $t \in \R$ such that $a=a'+t\unit_\B$. Now, there exists $b'\in\alg{D}$ such that $\|a'-b'\|_\B\leq\varepsilon$. Let $b = b'+t\unit_\B \in \alg{L}$. Then $\|a-b\|_\B = \|a'-b'\|_\B \leq\varepsilon$, so $\alg{L}_{f(n)}$ is included in an $\varepsilon$-neighborhood of $\alg{L}$. Using a symmetric argument, we conclude:
\begin{equation*}
\Haus{}\left(\alg{L}_{f(n)},\alg{L}\right) \leq\varepsilon
\end{equation*}
and thus $\left(\alg{L}_{f(n)}\right)_{n\in\N}$ converges to $\alg{L}$ for the Hausdorff distance $\Haus{}$. 

Moreover, $\alg{D} = \{a\in\alg{L} : \varphi(a) = 0\}$ by construction and continuity of $\varphi$. Last, as $\alg{D}$ is compact, hence closed, the set $\alg{L} = \alg{D} + \R\unit_\B$ is closed as well: if $(l_n)_{n\in\N} \in \alg{L}$ converges to some $l$ in $\B$ then $(l_n-\varphi(l_n)\unit_\B)_{n\in\N}$ is a sequence in $\alg{D}$, and thus by continuity of $\varphi$ and since $\alg{D}$ is closed, $l-\varphi(l)\unit_\B \in \alg{D}$. Thus $l\in\alg{L}$.
 
For all $b\in\sa{\B}$ we define:
\begin{equation*}
\Lip(b) = \inf \{ \lambda > 0 : b \in \lambda \alg{L} \}\text{.}
\end{equation*}

Using our assumptions on $F$, by Lemma (\ref{Haus-limit-lemma}), the set $\alg{L}$ satisfies Lemma (\ref{quasi-Leibniz-lemma}), and thus $\Lip$ is a $F$-quasi-Leibniz seminorm.

Certainly $\Lip$ may assume the value $\infty$. Let $\alg{J} = \dom{\Lip}$ be the set of self-adjoint elements in $\B$ for which $\Lip$ is finite.

If $a,b \in \alg{J}$ then:
\begin{equation*}
\Lip(\Jordan{a}{b}), \Lip(\Lie{a}{b}) \leq F(\|a\|_\A,\|b\|_\A,\Lip(a),\Lip(b)) < \infty
\end{equation*} 
so $\alg{J}$ is a Jordan-Lie subalgebra of $\sa{\B}$. We define:
\begin{equation*}
\A = \{ b \in \B : \Re(b),\Im(b) \in \alg{J} \}
\end{equation*}
and we check that $\A$ is a C*-subalgebra of $\B$ with the same unit as $\B$ and such that $\sa{\A} = \alg{J}$.

If $\Lip(a) = 0$ for some $a\in\alg{J}$, then we have $\Lip(a - \varphi(a)\unit_\A) = 0$ as well since $\Lip(\unit_\A) = 0$ and $\Lip$ is a seminorm by construction. Thus $a-\varphi(a)\unit_\A \in \alg{D}$. Now, for any $t\in\R$, we have $\varphi(t(a-\varphi(a)\unit_\A)) = 0$ and $\Lip(t(a-\varphi(a)\unit_\A)) = 0$, so $t(a-\varphi(a)\unit_\A)\in\alg{D}$ for all $t \in \R$. Since $\alg{D}$ is norm bounded, we conclude that $a=\varphi(a)\unit_\A$ as desired.

Since $\alg{D}\subseteq \alg{R}$, for any two states $\varphi,\psi\in\StateSpace(\A)$ and for all $a\in\alg{L}$ we have $|\varphi(a)-\psi(a)|\leq K$ and thus $\diam{\StateSpace(\A)}{\Kantorovich{\Lip}} \leq K$.

Moreover, since $\alg{D}$ is compact, we conclude that $\Lip$ is an $F$-quasi-Leibniz Lip-norm and $(\A,\Lip)$ is an {\Qqcms{F}} by Theorem (\ref{Rieffel-thm}).

We now prove that $\dim\A\leq d$. First, for all $n\in\N$, let $d_n = \dim_\R \sa{\A_{g(f(n))}}$. By assumption, since $\dim_\R\sa{\A_{g(f(n))}} = \dim_\C\A_{g(f(n))}$, we conclude that $d_n \in \{1,\ldots,d\}$. Thus, there exists a constant subsequence $(d_{f_1(n)})_{n\in\N}$ of $(d_n)_{n\in\N}$. Set $g_1 = g \circ f \circ f_1$ and $\delta = d_{f_1(0)} = d_{f_1(1)} = \ldots$. 

Now, we note that by construction, $\alg{J}$ is the linear span (over $\R$) of $\alg{L}$. Assume that we have a linearly independent family $(a_j)_{j\in\{1,\ldots,\delta+1\}}$ in $\alg{L}$. Since $\alg{L}$ is the limit, for the Hausdorff distance, of $(\alg{L}_{g(n)})_{n\in\N}$, we conclude that for all $j\in\{1,\ldots,\delta\}$, there exists a sequence $(a_j^n)_{n\in\N}$ with $a_j^n\in\alg{L}_{g(n)}$ for all $n\in\N$ and $\lim_{n\rightarrow\infty} a_j^n = a_j$. 

Fix $n\in\N$. Then, as $\dim{\A_{g(n)}} = \delta$, the family $(a_j^n)_{j\in\{1,\ldots,\delta+1\}}$ is linearly dependent. Thus there exists $\lambda_1^n,\ldots,\lambda_{\delta+1}^n \in \R$ with:
\begin{equation*}
\sum_{j=1}^{\delta+1} \lambda_j^n a_j^n = 0\text{ and }\max\{|\lambda_j^n|:j=1,\ldots,\delta+1\} = 1\text{.}
\end{equation*}
Now, in particular, the sequences $(\lambda_j^n)_{n\in\N}$ are bounded, for all $j$ in the finite set $\{1,\ldots,\delta+1\}$. Thus, there exists a strictly increasing function $g_2: \N\rightarrow\N$ such that for all $j\in\{1,\ldots,\delta+1\}$, the sequence $(\lambda_j^{g_2(n)})_{n\in\N}$ converges to some $\lambda_j$. Note that by continuity:
\begin{equation*}
\max\{|\lambda_j|:j=1,\ldots,\delta+1\} = 1\text{.}
\end{equation*}

On the other hand, again by continuity:
\begin{equation*}
\sum_{j=1}^{\delta+1} \lambda_j a_j = 0\text{.}
\end{equation*}
This last two equations contradicts that $(a_j)_{j\in\{1,\ldots,\delta+1}$ is linearly independent. Thus, the span of $\alg{L}$ is of dimension (over $\R$) at most $\delta$, thus at most $d$. Consequently, $\dim_\C\A = \dim_\R\mathrm{span}\alg{L} \leq d$ as desired.

Now, we wish to conclude by showing that $(\A_{g(f(n))},\Lip_{g(f(n))})_{n\in\N}$ converges to $(\A,\Lip)$ for the quantum propinquity. Let $\varepsilon > 0$. There exists $N\in\N$ such that for all $n\geq N$, we have $\Haus{}(\alg{L}_{g\circ f(n)},\alg{L})\leq \varepsilon$. Let now $n\geq N$.

For all $a\in\A_{g\circ f(n)}$ and $b\in \A$, we set:
\begin{equation*}
N_n(a,b) = \frac{1}{\varepsilon}\|a-b\|_\B\text{,}
\end{equation*}
and
\begin{equation*}
\Lip^n(a,b) = \max\left\{\Lip_{g\circ f(n)}(a), \Lip(b), N_n(a,b) \right\}\text{.}
\end{equation*}
It is easily checked that $N_n$ is a bridge in the sense of \cite[Definition 5.1]{Rieffel00}: in particular, if $a\in\A_{g\circ f(n)}$ with $\Lip_{g\circ f(n)}(a)\leq 1$ then there exists $b \in \alg{L}$ with $\|a-b\|_\B\leq\varepsilon$, which implies $\Lip^n(a,b)\leq 1$; similarly if $b\in \A$ with $\Lip(b)\leq 1$, i.e. $b\in\alg{L}$, then there exists $a\in\alg{L}_{f(n)}$ with $\|a-b\|_\B\leq\varepsilon$ and thus $\Lip^n(a,b) \leq 1$.

Hence by \cite[Theorem 5.2]{Rieffel00}, the seminorm $\Lip^n$ is a Lip-norm. It is lower semi-continuous by construction, and it is easily checked that $\Lip^n$ is $F$-quasi-Leibniz, as in our proof of Theorem (\ref{tunnel-composition-thm}). 

Let $\tau_n = (\A_{g\circ f(n)} \oplus \A, \Lip, \rho_n, \rho)$ with $\rho_n : \A_{g\circ f(n)} \oplus \A \twoheadrightarrow \A_{g\circ f(n)}$ and $\rho : \A_{g\circ f(n)}\oplus\A\twoheadrightarrow \A$ the two canonical surjections. By construction, $\tau_n$ is an $F$-tunnel.

The depth of $\tau_n$ (Definition (\ref{length-def})) is null, and thus the length of $\tau_n$ is its spread. Now, let $\mu\in\StateSpace(\A_{g\circ f(n)})$. Let $\mu'$ be a state of $\B$ given by applying the Hahn-Banach Theorem for positive linear maps to $\mu$. Let $\nu$ be the restriction of $\mu$ to $\A$ and note that $\nu \in\StateSpace(\A)$ as $\A$ and $\B$ share the same unit. If $(a,b)\in \A_{g\circ f(n)}\oplus\A$ with $\Lip^n(a,b)\leq 1$, then:
\begin{equation*}
|\mu(a)-\nu(b)| \leq \|a-b\|_\B \leq \varepsilon\text{.}
\end{equation*}
We may perform the same computation with $\A_{g\circ f(n)}$ and $\A$ switched. Thus $\tunnellength{\tau_3} \leq \varepsilon$. Thus $\tunnelextent{\tau_n}\leq 2\varepsilon$.

Consequently: for all $\varepsilon > 0$, there exists $N\in\N$ such that for all $n\geq N$, we have:
\begin{equation*}
\propinquity{F}((\A_{g\circ f(n)},\Lip_{g\circ f(n)}),(\A,\Lip)) \leq 2\varepsilon\text{.}
\end{equation*}

Moreover, $(\A,\Lip)$ is a {\Qqcms{F}} of diameter at most $K$ and dimension at most $d$. This completes our proof.
\end{proof}

We conclude by observing that, using the notations of the proof of Theorem (\ref{f-compact-thm}), the quadruple $(\B, \unit_\B, \iota_n,\iota)$ where $\iota_n : \A_{g\circ f(n)}\hookrightarrow\B$ and $\iota:\A\hookrightarrow\B$ are the inclusion maps, is a bridge for the extension of the quantum propinquity to $F$-{\gQqcms s}. Thus, Theorem (\ref{f-compact-thm}) is also valid for the quantum propinquity.

\section{A compactness Theorem}

This section establishes the core result of our paper, which characterizes compact classes of {\gQqcms s} for the dual propinquity among all subclasses of the closure of finite dimensional {\gQqcms s}. Our reason for working within this closure is that our main theorem employs the following key notion, motivated by Gromov's Theorem (\ref{Gromov-Compactness-thm}):

\begin{definition}
Let $F$ be a permissible function, and let $\mathcal{QQCMS}_{F}$ be the class of all {\Qqcms{F}s}. Let $(\A,\Lip)$ be a {\Qqcms{F}} and let $\varepsilon > 0$. 

The \emph{covering number} $\covn{F}{\A,\Lip}{\varepsilon}$ is:
\begin{equation*}
\covn{F}{\A,\Lip}{\varepsilon} = \min\left\{ n \in \N : \begin{array}{l} \exists (\B,\Lip_\B) \in \mathcal{QQMS}_{F} \text{ such that}\\
\dim_\C \B \leq n \text{ and}\\
\propinquity{F}((\A,\Lip_\A),(\B,\Lip_\B))\leq\varepsilon 
\end{array}
\right\}\text{,}
\end{equation*}
where by convention, $\min\emptyset = \infty$.
\end{definition}

Our generalization of Gromov's Compactness Theorem (\ref{Gromov-Compactness-thm}) is given by:

\begin{theorem}\label{compactness-thm}
Let $F$ be a strongly permissible function. and let $\mathcal{A}$ be a class of {\Qqcms{F}s} in the closure of the finite dimensional {\Qqcms{F}s} for the dual propinquity $\propinquity{F}$. The following assertions are equivalent:
\begin{enumerate}
\item $\mathcal{A}$ is totally bounded for $\propinquity{F}$,
\item there exists a function $f : (0,\infty)\rightarrow \N$ and $K > 0$ such that for all $(\A,\Lip_\A) \in \mathcal{A}$, we have:
\begin{itemize}
\item $\diam{\A}{\Lip_\A}\leq K$,
\item for all $\varepsilon > 0$ we have $\covn{F}{\A,\Lip_\A}{\varepsilon}\leq f(\varepsilon)$.
\end{itemize}
\end{enumerate}
\end{theorem}

\begin{proof}
Assume (2), i.e. assume that there exists $f : (0,\infty)\rightarrow\N$ and $K>0$ such that for all $(\A,\Lip_\A)\in\mathcal{A}$ and $\varepsilon > 0$, we have $\diam{\StateSpace(\A)}{\Kantorovich{\Lip_\A}} \leq K$ and:
\begin{equation*}
\covn{F}{\A,\Lip_\A}{\varepsilon} \leq f(\varepsilon)\text{.}
\end{equation*}

Let $\varepsilon > 0$. 

First, we note that if $(\A,\Lip_\A) \in \mathcal{A}$, then there exists a {\Qqcms{F}} $(\alg{a},\Lip_{\alg{a}})$ such that:
\begin{itemize}
\item $\dim_\C \alg{a} \leq f\left(\frac{\varepsilon}{3}\right)$,
\item $\propinquity{F}((\alg{a},\Lip_{\alg{a}}),(\A,\Lip_\A)) \leq \frac{\varepsilon}{3}$.
\end{itemize}
Consequently, we note that $\diam{\alg{a}}{\Lip_{\alg{a}}} \leq K + \frac{2\varepsilon}{3}$, since the function $(\B,\Lip_\B)\in\mathcal{A}\mapsto \diam{\B}{\Lip_\B}$ is $2$-Lipschitz for the quantum Gromov-Hausdorff distance, and thus for the dual propinquity, by \cite[Lemma 13.6]{Rieffel00}.

Now, by Theorem (\ref{f-compact-thm}), the class:
\begin{equation*}
\mathcal{F}_\varepsilon = \left\{ (\B,\Lip) \in \mathcal{QQCMS}_{F} : \dim_\C\B \leq f\left(\frac{\varepsilon}{3}\right) \text{ and } \diam{\B}{\Lip}\leq K + \frac{2\varepsilon}{3} \right\}
\end{equation*}
is compact for $\propinquity{F}$. 

Let:
\begin{equation*}
\mathcal{G}_\varepsilon = \left\{ (\B,\Lip_\B)\in \mathcal{F}_\varepsilon : \exists (\A,\Lip_\A) \in \mathcal{A} \quad \propinquity{F}((\A,\Lip_\A), (\B,\Lip_\B)) \leq \frac{\varepsilon}{3} \right\}\text{.}
\end{equation*}

Since $\mathcal{G}_\varepsilon\subseteq \mathcal{F}_\varepsilon$, we conclude that $\mathcal{G}_\varepsilon$ is totally bounded for the dual propinquity. Thus, there exists a finite subset $\mathcal{J}_\varepsilon$ of $\mathcal{G}_\varepsilon$ which is $\frac{\varepsilon}{3}$ dense in $\mathcal{G}_\varepsilon$ for $\propinquity{F}$.

Therefore, up to invoking choice, there exists a finite subset $\mathcal{A}_\varepsilon$ of $\mathcal{A}$ such that for all $(\B,\Lip_\B)\in\mathcal{J}_\varepsilon$ there exists $(\A,\Lip_\A)\in\mathcal{A}_\varepsilon$ such that $\propinquity{F}((\A,\Lip_\A),(\B,\Lip_\B))\leq\frac{\varepsilon}{3}$. 

Now, let $(\A,\Lip_\A) \in \mathcal{A}$. There exists $(\alg{a},\Lip_{\alg{a}}) \in \mathcal{G}_\varepsilon$ such that:
\begin{equation*}
\propinquity{F}((\alg{a},\Lip_{\alg{a}}),(\A,\Lip_\A))\leq\frac{\varepsilon}{3}\text{.}
\end{equation*}
Now, there exists $(\alg{b},\Lip_{\alg{b}})\in\mathcal{J}_\varepsilon$ such that:
\begin{equation*}
\propinquity{F}((\alg{a},\Lip_{\alg{a}}),(\alg{b},\Lip_{\alg{b}}))\leq\frac{\varepsilon}{3}\text{.}
\end{equation*}
Last, by our choice, there exists $(\B,\Lip_{\B})\in\mathcal{A}_\varepsilon$ with:
\begin{equation*}
\propinquity{F}((\B,\Lip_{\B}),(\alg{b},\Lip_{\alg{b}}))\leq\frac{\varepsilon}{3}\text{.}
\end{equation*}

Consequently:
\begin{equation*}
\begin{split}
\propinquity{F}&((\A,\Lip_\A),(\B,\Lip_\B)) \\
&\leq \propinquity{F}((\A,\Lip_\A),(\alg{a},\Lip_{\alg{a}})) + \propinquity{F}((\alg{a},\Lip_{\alg{a}}),(\alg{b},\Lip_{\alg{b}})) + \propinquity{F}((\alg{b},\Lip_{\alg{b}}), (\B,\Lip_{\B}))\\
&\leq \varepsilon\text{.}
\end{split}
\end{equation*}

Thus $\mathcal{A}_\varepsilon$ is $\varepsilon$-dense in $\mathcal{A}$ for the dual propinquity, and is a finite set. Thus, $\mathcal{A}$ is totally bounded for $\propinquity{F}$.

Assume (1) now, i.e. assume that $\mathcal{A}$ is totally bounded. Since the function:
\begin{equation*}
(\A,\Lip)\in\mathcal{A} \mapsto \diam{\A}{\Lip}
\end{equation*}
is $2$-Lipschitz for the quantum Gromov-Hausdorff distance by \cite[Lemma 13.6]{Rieffel00}, it is continuous, and thus it is bounded above since $\mathcal{A}$ is totally bounded.

Now, let $\varepsilon > 0$. Since $\mathcal{A}$ is totally bounded, there exists a finite subset $\mathfrak{A}_\varepsilon$ of $\mathcal{A}$ which is $\frac{\varepsilon}{2}$-dense in $\mathcal{A}$ for $\propinquity{F}$. For each $(\A,\Lip)\in\mathcal{A}_\varepsilon$, there exists a finite dimensional {\Qqcms{F}} $\alg{f}(\A,\Lip)$ such that:
\begin{equation*}
\propinquity{F}((\A,\Lip),\alg{f}(\A,\Lip))\leq\frac{\varepsilon}{2}
\end{equation*}
by assumption on $\mathcal{A}$.

Let:
\begin{equation*}
f(\varepsilon) = \max \left\{ \dim_\C \alg{f}(\A,\Lip) : (\A,\Lip) \in \mathcal{A}_\varepsilon \right\}\text{.}
\end{equation*}
If $(\A,\Lip)\in\mathcal{A}$, then there exists $(\B,\Lip_\B) \in \mathcal{A}_\varepsilon$ such that $\propinquity{F}((\A,\Lip),(\B,\Lip_\B))\leq\frac{\varepsilon}{2}$. Thus:
\begin{equation*}
\propinquity{F}((\A,\Lip),\alg{f}(\B,\Lip_\B)) \leq \varepsilon\text{.}
\end{equation*}
Thus $\covn{F}{\A,\Lip}{\varepsilon} \leq f(\varepsilon)$ by definition.

This completes our proof.
\end{proof}

We thus conclude:

\begin{corollary}
Let $F$ be a strongly permissible function, and let $\mathcal{A}$ be a class of {\Qqcms{F}s} which lies within the closure of the finite dimensional {\Qqcms{F}s} for the dual propinquity $\propinquity{F}$.

The class $\mathcal{A}$ is compact for the dual propinquity $\propinquity{F}$ if and only if there exists $f : (0,\infty)\rightarrow\N$ and $K > 0$ such that:
\begin{enumerate}
\item for all $(\A,\Lip)\in\mathcal{A}$ we have $\diam{\A}{\Lip}\leq K$,
\item for all $(\A,\Lip)\in\mathcal{A}$ and for all $\varepsilon > 0$, we have:
\begin{equation*}
\covn{F}{\A,\Lip}{\varepsilon} \leq f(\varepsilon)
\end{equation*}
\item $\mathcal{A}$ is closed for $\propinquity{F}$.
\end{enumerate}
\end{corollary}

\begin{proof}
By Theorem (\ref{prop-thm}), the dual propinquity $\propinquity{F}$ is complete (as $F$ is continuous). The result then follows from Theorem (\ref{compactness-thm}).
\end{proof}

Thus, we are led to study the closure of finite dimensional {\gQqcms s} for the dual propinquity. This question proves tricky. However, our use of quasi-Leibniz Lip-norms, instead of Leibniz Lip-norms, allows us to establish that a large class of {\Qqcms{(C,D)}s} lie within the closure of the class of finite dimensional {\Qqcms{(C',D')}s} for arbitrary $C'>C\geq 1$ and $D'>D\geq 0$. This will be the subject of the next section of this paper.

\section{Finite Dimensional Approximations for pseudo-diagonal {\gQqcms s}}

For an arbitrary {\Lqcms} $(\A,\Lip)$ and $\varepsilon > 0$, it is not immediately clear how to find a finite dimensional C*-algebra $\B$ and a Leibniz Lip-norm $\Lip_\B$ on $\B$ such that $\propinquity{}((\A,\Lip),(\B,\Lip_\B)) \leq\varepsilon$. In this section, we propose to work with {\Lqcms s} whose underlying C*-algebra provides natural topological, finite dimensional approximations, and then attempt to construct Lip-norms on these approximations. As we shall see, however, these Lip-norms will be quasi-Leibniz. Yet, for any $C>1$ and $D>0$,  and for any $\varepsilon$, we will show that $(\A,\Lip)$ is within $\varepsilon$-distance of some {\Qqcms{(C,D)}} of finite dimension, for the dual propinquity, as long as $\A$ satisfies a weakened form of nuclear quasi-diagonality, introduced below as pseudo-diagonality. Our result requires us to work within the framework of {\gQqcms s} and thus, we take this opportunity to establish the result of this section in greater generality: any pseudo-diagonal {\Qqcms{(C,D)}} is the limit of finite dimensional {\Qqcms{(C',D')}s} for $C'> C$ and $D'>D$.

We begin with our choice of a class of C*-algebras with a natural notion of topological finite dimensional approximation, relevant for our purpose:

\begin{definition}\label{pseudo-diagonal-def}
A unital C*-algebra $\A$ is \emph{pseudo-diagonal}, when for all $\varepsilon > 0$ and for all finite subset $\alg{F}$ of $\A$, there exist a finite dimensional C*-algebra $\B$ and two positive unital linear maps $\varphi : \B\rightarrow\A$ and $\psi:\A\rightarrow\B$ such that:
\begin{enumerate}
\item for all $a\in\alg{F}$, we have $\|a-\varphi\circ\psi(a)\|_\A\leq\varepsilon$,
\item for all $a,b \in \alg{F}$, we have $\|\psi\left(\Jordan{a}{b}\right)-\Jordan{\psi(a)}{\psi(b)}\|_\B\leq\varepsilon$,
\item for all $a,b \in \alg{F}$, we have $\|\psi\left(\Lie{a}{b}\right)-\Lie{\psi(a)}{\psi(b)}\|_\B\leq\varepsilon$.
\end{enumerate}
\end{definition}

Note that unital positive linear maps have norm $1$, and thus are contractions.

Definition (\ref{pseudo-diagonal-def}) is motivated by a characterization of unital nuclear, quasi-diagonal C*-algebras, due to Blackadar and Kirchberg \cite[Theorem 5.2.2]{Blackadar97}:

\begin{theorem}[Theorem 5.2.2, \cite{Blackadar97}]\label{BK-thm}
A unital C*-algebra $\A$ is \emph{nuclear Quasi-Diagonal} if and only if, for all $\varepsilon > 0$ and all finite subsets $\alg{F}\subseteq \alg{A}$, there exist a finite dimensional C*-algebra $\B$ and two contractive completely positive maps $\psi: \A \rightarrow \B$ and $\varphi : \B\rightarrow\A$ such that:
\begin{enumerate}
\item for all $a,b\in\alg{F}$ we have $\|\psi(ab)-\psi(a)\psi(b)\|_{\B} \leq \varepsilon$,
\item for all $a\in\alg{F}$ we have $\|a-\varphi\circ\psi(a)\|_\A\leq\varepsilon$,
\end{enumerate}
\end{theorem}

\begin{remark}\label{useful-rmk}
With the notations of Theorem (\ref{BK-thm}), if $a\in\alg{F}$ then $\|a\|_\A=\|a-\varphi(\psi(a))\|_\A + \|\varphi(\psi(a))\|_\A \leq \varepsilon + \|\psi(a)\|_\B$, i.e. 
\begin{equation*}
\forall a \in \alg{F}\quad \|a\|_\A\leq\|\psi(a)\|_\B + \varepsilon\text{.}
\end{equation*}

This property, together with the first assertion of Theorem (\ref{BK-thm}), characterizes quasi-diagonal C*-algebras, as proved by D. Voiculescu \cite{Voiculescu91}. The existence of the map $\varphi$, in turn, adds the nuclearity property.
\end{remark}

While Theorem (\ref{BK-thm}) involves completely positive maps, Definition (\ref{pseudo-diagonal-def}) only requires positive maps, as this suffices for our proof of finite approximation in the sense of the dual propinquity. On the other hand, Theorem (\ref{BK-thm}) does not provide unital maps, while Definition (\ref{pseudo-diagonal-def}) requires them, as we will find this useful --- for instance, to map states to states. Nonetheless, unital nuclear quasi-diagonal C*-algebras are unital pseudo-diagonal. To this end, we show that we can, in fact, require the maps in Theorem (\ref{BK-thm}) to be unital (see \cite{smith04} for a similar problem, albeit a different construction which would fail to preserve our almost-multiplicative property). We begin with:

\begin{lemma}\label{unital-correction-lemma}
Let $\A$ be a unital C*-algebra such that, for every nonempty finite subset $\alg{F}$ of $\A$ and for every $\varepsilon > 0$, there exists a finite dimensional C*-algebra $\B_\alpha$ and two positive contractions $\varphi_\varepsilon : \B_\alpha \rightarrow \A$ and $\psi_\varepsilon : \A\rightarrow \B_\alpha$ such that for all $x,y\in\alg{F}$:
\begin{equation}\label{app-eq1}
\|x - \varphi_\varepsilon\circ\psi_\varepsilon(x) \|_\A \leq \varepsilon
\end{equation} 
and
\begin{equation}\label{app-eq2}
\|\psi_\varepsilon(x) \psi_\varepsilon(y) - \psi_\varepsilon(x y)\|_{\B_\alpha} \leq \varepsilon\text{.}
\end{equation}

Then for all $\varepsilon > 0$ and for all finite subset $\alg{F}\subseteq \A$, there exists a finite dimensional C*-algebra $\D_\varepsilon$  with unit $\unit'_\varepsilon$ and two positive contractive map $\varsigma_\varepsilon : \A \rightarrow \D_\varepsilon$ and $\vartheta_\varepsilon : \D_\varepsilon \rightarrow \A$ such that, for all $x,y \in \alg{F}$:
\begin{equation*}
\|x - \vartheta_\varepsilon\circ\varsigma_\varepsilon(x)\|_\A \leq  \varepsilon
\end{equation*}
and
\begin{equation*}
\|\varsigma_\varepsilon(x)\varsigma_\varepsilon(y) - \varsigma_\varepsilon(xy)\|_{\D_\varepsilon} \leq  \varepsilon
\end{equation*}
while:
\begin{equation*}
\|\unit'_\varepsilon - \varsigma_\varepsilon(\unit_\A)\|_{\D_\varepsilon} \leq \varepsilon\text{.}
\end{equation*}
\end{lemma}

\begin{proof}
Let $\alg{F}$ be finite subset of $\A$. Let $\alg{F}_1 = \alg{F}\cup\{\unit_\A\}$. For all $\varepsilon \in (0,1)$, by assumption, there exist a finite dimensional C*-algebra $\B_\varepsilon$ and two positive contractions $\varphi_\varepsilon : \B_\varepsilon \rightarrow \A$ and $\psi_\varepsilon : \A\rightarrow\B_\varepsilon$ such that, for all $a,b \in \alg{F}_1$:
\begin{enumerate}
\item $\|\psi_\varepsilon(ab) - \psi_\varepsilon(a)\psi_\varepsilon(b)\|_{\B_\varepsilon} \leq \varepsilon - \varepsilon^2$,
\item $\|a-\varphi_\varepsilon\circ\psi_\varepsilon(a)\|_\A \leq\varepsilon - \varepsilon^2$.
\end{enumerate}

We shall henceforth tacitly identify $\B_\varepsilon$ with some subalgebra of a full matrix algebra $\alg{M}_\varepsilon$.

Let $\varepsilon \in \left(0, \frac{1}{4}\right)$ and let:
\begin{equation*}
\epsilon = \frac{\varepsilon}{3 + 2 \max\{\|x\|_\A\|y\|_\A : x,y \in\alg{F}_1 \} }\text{.}
\end{equation*}
By construction, $\epsilon \in \left(0,\frac{1}{4}\right)$.

Our first observation is that $\psi_\epsilon(\unit_\A)$ is a positive operator of norm at most $1$ since $\psi_\epsilon$ is a positive contraction. Thus there exists a unitary $u$ in $\alg{M}_\epsilon$ such that $u\psi_\epsilon(\unit_\A)u^\ast$ is a positive diagonal matrix less than the identity. Now, if we replace $\psi_\epsilon$ with $u \psi_\epsilon(\cdot) u^\ast$ and $\varphi_\epsilon$ with $\varphi_\epsilon(u^\ast \cdot u)$, then we obtain a pair of maps which also satisfy Assertions (\ref{app-eq1}) and (\ref{app-eq2}) (up to replacing $\B_\epsilon$ with $u \B_\epsilon u^\ast$). We shall henceforth assume we have made this change.

Thus, $\psi_\epsilon(\unit_\A)$ is a diagonal matrix which we denote by $D$ such that $0\leq D\leq 1$. Now, by construction, $\|\psi_\epsilon(\unit_\A) - \psi_\epsilon(\unit_\A)^2\|_{\B_\epsilon} \leq \epsilon - \epsilon^2 < \frac{1}{4}$. Hence we conclude that the spectrum $\sigma(D)$ of $D$ is a compact subset of $[0,\epsilon]\cup[1-\epsilon,1]$; moreover $[0,\epsilon]\cap[1-\epsilon,1] = \emptyset$. 

Let $P$ be the projection on the sum of the spectral subspaces of $D$ associated with the eigenvalues in $[1-\epsilon,1]$. Thus $1-P$ is the projection on the sum of the spectral subspaces of $D$ associated with eigenvalues in $[0,\epsilon]$.

Let now $x\in\alg{F}_1$. Then:
\begin{equation*}
\begin{split}
\|\psi_\epsilon(x) &- D\psi_\epsilon(x)D\|_{\B_\epsilon} \\
&\leq \|\psi_\epsilon(x)-\psi_\epsilon(x)D\|_{\B_\epsilon} + \|\psi_\epsilon(x)D - D\psi_\epsilon(x)D\|_{\B_\epsilon}\\
&= \|\psi_\epsilon(x\unit_\A) - \psi_\epsilon(x)\psi_\epsilon(\unit_\A)\|_{\B_\epsilon} + \|\left(\psi_\epsilon(\unit_\A x)-\psi_\epsilon(\unit_\A)\psi_\epsilon(x)\right)D\|_{\B_\epsilon}\\
&\leq \epsilon + \epsilon \|D\|_{\B_\epsilon} \leq 2\epsilon\text{.}
\end{split}
\end{equation*}
Moreover:
\begin{equation*}
\begin{split}
\|D\psi_\epsilon(x)D - P\psi_\epsilon(x)P\|_{\B_\epsilon} &= \|(D-P)\psi_\epsilon(x)D\|_{\B_\epsilon} + \|P\psi_\epsilon(x)(D-P)\|_{\B_\epsilon}\\
&\leq 2\|D-P\|_{\B_\epsilon} \|x\|_\A \text{.}
\end{split}
\end{equation*}
Now by construction, $\|D-P\|_{\B_\epsilon} \leq\epsilon$. Thus:
\begin{equation}\label{unital-correction-lemma-eq1}
\begin{split}
\|\psi_\epsilon(x) - P\psi_\epsilon(x)P\|_{\B_\epsilon} &\leq \|\psi_\epsilon(x) - D\psi_\epsilon(x)D\|_{\B_\epsilon} + \|D\psi_\epsilon(x)D - P\psi_\epsilon(x)P\|_{\B_\epsilon}\\
&\leq 2\epsilon(1+\|x\|_\A)\text{.}
\end{split}
\end{equation}

The map $\varsigma_\varepsilon = P\psi_\epsilon(\cdot)P$ is a positive contraction by construction. We set $\vartheta_\varepsilon = \varphi_\epsilon$. We now check that our construction leads to the desired conclusion, with $\D_\varepsilon = P\B_\epsilon P$. We rename $P$ as $\unit'_\varepsilon$.
We have:
\begin{equation*}
\|\unit_\varepsilon' - \varsigma_\varepsilon(\unit_\A)\|_{\B_\epsilon} = \|P - PDP\|_{\B_\epsilon} \leq \epsilon \leq \varepsilon  \text{.}
\end{equation*}

Now let $x\in\alg{F}_1$. Using Inequality (\ref{unital-correction-lemma-eq1}):
\begin{equation*}
\begin{split}
\|x-\vartheta_\varepsilon(\varsigma_\varepsilon(x))\|_\A &= \|x-\varphi_\epsilon(\varsigma_\varepsilon(x))\|_\A \\
&= \|x-\varphi_\epsilon(P\psi_\epsilon(x)P)\|_\A\\
&\leq \|x-\varphi_\epsilon(\psi_\epsilon(x))\|_\A + \|\varphi_\epsilon(\psi_\epsilon(x)) - \varphi_\epsilon(P\psi_\epsilon(x)P)\|_\A\\
&\leq \epsilon + \|\psi_\epsilon(x) - P\psi_\epsilon(x)P\|_{\B_\epsilon} \\
&\leq \epsilon(3 + 2\|x\|_\A) \leq \varepsilon \text{.}
\end{split}
\end{equation*}

Last, let $x,y \in \alg{F}_1$. Since $\unit_\A\in\alg{F}_1$ as well, we obtain:
\begin{equation*}
\begin{split}
\|\varsigma_\varepsilon(x)\varsigma_\varepsilon(y) &- \varsigma_\varepsilon(xy)\|_{\B_\epsilon}\\
&= \|P\psi_\epsilon(x)P\psi_\epsilon(y)P - P\psi_\epsilon(xy)P\|_{\B_\epsilon}\\
&\leq \|\psi_\epsilon(x)P\psi_\epsilon(y) - \psi_\epsilon(xy)\|_{\B_\epsilon}\\
&\leq \|\psi_\epsilon(x)P\psi_\epsilon(y) - \psi_\epsilon(x)\psi_\epsilon(y)\|_{\B_\epsilon} + \|\psi_\epsilon(x)\psi_\epsilon(y)-\psi_\epsilon(xy)\|_{\B_\epsilon}  \\
&\leq \|\psi_\epsilon(x)(P-D)\psi_\epsilon(y)\|_{\B_\epsilon} + \|\psi_\epsilon(x)D\psi(y)-\psi_\epsilon(x)\psi_\epsilon(y)\|_{\B_\epsilon} + \epsilon \\
&\leq \epsilon \|x\|_{\A}\|y\|_{\A} + \|\psi_\epsilon(x)\psi_\epsilon(\unit_\A)\psi_\epsilon(y) - \psi_\epsilon(x\unit_\A)\psi_\epsilon(y)\|_{\B_\epsilon} + \epsilon\\
&\leq \epsilon\|x\|_\A\|y\|_\A + \epsilon\|y\|_\A + \epsilon = \epsilon(1+\|y\|_\A+\|x\|_\A\|y\|_\A) \leq \varepsilon \text{.}
\end{split}
\end{equation*}
Note (although not needed for our proof) that by symmetry, grouping $\unit_\A$ with $y$ instead of $x$ above, we get:
\begin{equation*}
\|\varsigma_\varepsilon(x)\varsigma_\varepsilon(y) - \varsigma_\varepsilon(xy)\|_{\B_\epsilon} \leq \epsilon(1+\min\{\|x\|_\A,\|y\|_\A\} + \|x\|_\A\|y\|_\A) \text{.}
\end{equation*}

This concludes our proof.
\end{proof}

\begin{corollary}
A unital nuclear quasi-diagonal C*-algebra $\A$ is a unital pseudo-diagonal C*-algebra.
\end{corollary}

\begin{proof}
Let $\alg{F}$ be a finite subset of $\A$, which we assume without loss of generality, contains $\unit_\A$. Let $\varepsilon < \frac{1}{4}$. By Lemma (\ref{unital-correction-lemma}), there exist a finite dimensional C*-algebra $\B_\varepsilon$ and two positive contractions $\varphi_\varepsilon : \B_\varepsilon \rightarrow \A$ and $\psi_\varepsilon : \A\rightarrow\B_\varepsilon$ such that:
\begin{enumerate}
\item for all $x\in\alg{F}$ we have $\|x - \varphi_\varepsilon\circ\psi_\varepsilon(x)\|_\A \leq \varepsilon$,
\item for all $x,y \in \alg{F}$ we have $\|\psi_\varepsilon(xy)-\psi_\varepsilon(x)\psi_\varepsilon(y)\|_{\B_\varepsilon} \leq\varepsilon$,
\item we have $\|\unit_\varepsilon - \psi_\varepsilon(\unit_\A)\|_{\B_\varepsilon} \leq \varepsilon$, where $\unit_\varepsilon$ is the unit of $\B_\varepsilon$.
\end{enumerate}

Since $\|\psi_\varepsilon(\unit_\A) - \unit_\varepsilon\|_{\B_\varepsilon} < 1$, the matrix $\psi_\varepsilon(\unit_\A)$ is invertible, while it is also positive, of course.

We also note that:
\begin{equation*}
\begin{split}
\|\unit_\A - \varphi_\varepsilon(\unit_\A)\|_{\A} &= \|\unit_\A - \varphi_\varepsilon\circ\psi_\varepsilon(\unit_\A)\|_{\B_\varepsilon} + \|\varphi_\varepsilon\circ\psi_\varepsilon(\unit_\A) - \varphi_\varepsilon(\unit_\varepsilon)\|_\A\\
&\leq \varepsilon + \|\psi_\varepsilon(\unit_\A) - \unit_\varepsilon\|_{\B_\varepsilon} \\
&\leq 2\varepsilon < 1 \text{.}
\end{split}
\end{equation*}
Thus $\varphi(\unit_\A)$ is also invertible in $\A$ --- and of course also positive.

Define:
\begin{equation*}
\theta_\varepsilon = \varphi_\varepsilon(\unit_\varepsilon)^{-\frac{1}{2}} \varphi_\varepsilon(\cdot) \varphi_\varepsilon(\unit_\varepsilon)^{-\frac{1}{2}} \text{ and }\varsigma_\varepsilon = \psi_\varepsilon(\unit_\A)^{-\frac{1}{2}} \psi_\varepsilon(\cdot) \psi(\unit_\A)^{-\frac{1}{2}}\text{.}
\end{equation*}

We first note that $\theta_\varepsilon$ and $\varsigma_\varepsilon$ are both positive linear maps. Moreover, $\theta_\varepsilon$ and $\varsigma_\varepsilon$ are both unital by construction, so they are contractions as well.

Now, let $a,b \in \alg{F}$. Then:
\begin{equation*}
\begin{split}
\|\varsigma_\varepsilon(ab) &- \varsigma_\varepsilon(a)\varsigma_\varepsilon(b)\|_{\B_\varepsilon} \\
&\leq \|\psi_\varepsilon(\unit_\A)^{-\frac{1}{2}}\|^2_{\B_\varepsilon} \|\psi_\varepsilon(ab) - \psi_\varepsilon(a)\psi_\varepsilon^{-1}(\unit_\A)\psi_\varepsilon(b)\|_{\B_\varepsilon} \\
&\leq \|\psi_\varepsilon(\unit_\A)^{-\frac{1}{2}}\|^2_{\B_\varepsilon} \\
&\left(\| \psi_\varepsilon(a)\left(\unit_\varepsilon - \psi_\varepsilon^{-1}(\unit_\A)\right)\psi_\varepsilon(b)\|_{\B_\varepsilon} + \|\psi_\varepsilon(ab) - \psi_\varepsilon(a)\psi_\varepsilon(b)\|_{\B_\varepsilon}\right)\\
&\leq \|\psi_\varepsilon(\unit_\A)^{-\frac{1}{2}}\|^2_{\B_\varepsilon}\left(\left\|\unit_\varepsilon - \psi_\varepsilon^{-1}(\unit_\A)\right\|_{\B_\varepsilon}\|a\|_\A\|b\|_\A + \|\psi_\varepsilon(ab)-\psi_\varepsilon(a)\psi_\varepsilon(b)\|_{\B_\varepsilon}\right)
\end{split}
\end{equation*}
since $\psi_\varepsilon$ is a contraction. Now, the spectrum of the self-adjoint element $\psi_\varepsilon(\unit_\A)$ is a subset of $[1-\varepsilon,1]$ since $\|\psi_\varepsilon(\unit_\A)-\unit_\varepsilon\|_{\B_\varepsilon} \leq \varepsilon$. Consequently, as $\varepsilon < 1$, by the functional mapping theorem, the spectrum of $\psi_\varepsilon(\unit_\A)^{-1}-\unit_\varepsilon$ is included in $[0,\frac{1}{1-\varepsilon}-1]$. Thus, as $\psi_\varepsilon(\unit_\A)^{-1}-\unit_\varepsilon$ is a self-adjoint element:
\begin{equation*}
\begin{split}
\left\|\psi_\varepsilon(\unit_\A)^{-1} - \unit_\varepsilon \right\|_{\B_\varepsilon} \leq \frac{1}{1-\varepsilon} - 1\text{,}
\end{split}
\end{equation*}
and thus $\lim_{\varepsilon\rightarrow 0} \|\psi_\varepsilon(\unit_\A)^{-1} - \unit_\varepsilon\|_{\B_\varepsilon} = 0$. Similarly, $\lim_{\varepsilon\rightarrow 0} \|\unit_\varepsilon - \psi_\varepsilon(\unit_\A)^{-\frac{1}{2}}\|_{\B_\varepsilon} = 0$, and thus in particular,  $\lim_{\varepsilon\rightarrow 0} \|\psi_\varepsilon(\unit_\A)^{-\frac{1}{2}}\|_{\B_\varepsilon} = 1$.

Since $\lim_{\varepsilon\rightarrow 0}\|\psi_\varepsilon(ab)-\psi_\varepsilon(a)\psi_\varepsilon(b)\|_{\B_\varepsilon} = 0$, we obtain:
\begin{equation*}
\lim_{\varepsilon\rightarrow 0} \|\varsigma_\varepsilon(ab) - \varsigma_\varepsilon(a)\varsigma_\varepsilon(b)\|_{\B_\varepsilon} = 0\text{.}
\end{equation*}
Similarly, for all $a\in\alg{F}$, we have:
\begin{equation*}
\begin{split}
\|a &- \theta_\varepsilon\circ\varsigma_\varepsilon(a)\|_\A \\
&= \|a-\varphi_\varepsilon(\unit_\varepsilon)^{-\frac{1}{2}}\varphi_\varepsilon\left(\psi_\varepsilon(\unit_\A)^{-\frac{1}{2}}\psi_\varepsilon(a)\psi_\varepsilon^{-\frac{1}{2}}(\unit_\A)\right)\varphi_\varepsilon(\unit_\varepsilon)^{-\frac{1}{2}}\|_\A\\
&\leq \|a-\varphi_\varepsilon\circ\psi_\varepsilon(a)\|_\A \\
&\quad + \left\|\varphi_\varepsilon\left(\psi_\varepsilon(\unit_\A)^{-\frac{1}{2}}\psi_\varepsilon(a)\left(\unit_\varepsilon - \psi_\varepsilon(\unit_\A)^{-\frac{1}{2}}\right)\right) \right\|_\A \\
&\quad + \left\|\varphi_\varepsilon\left(\left(\unit_\varepsilon - \psi_\varepsilon(\unit_\A)^{-\frac{1}{2}}\right)\psi_\varepsilon(a)\right)\right\|_\A\\
&\quad + \left\|\left(\unit_\A-\varphi_\varepsilon(\unit_\varepsilon)^{-\frac{1}{2}}\right)\varphi_\varepsilon\left(\psi_\varepsilon(\unit_\A)^{-\frac{1}{2}}\psi_\varepsilon(a)\psi_\varepsilon^{-\frac{1}{2}}(\unit_\A)\right)\right\|_\A\\
&\quad + \left\|\varphi_\varepsilon(\unit_\varepsilon)^{-\frac{1}{2}}\varphi_\varepsilon\left(\psi_\varepsilon(\unit_\A)^{-\frac{1}{2}}\psi_\varepsilon(a)\psi_\varepsilon^{-\frac{1}{2}}(\unit_\A)\right)\left(\unit_\A-\varphi_\varepsilon(\unit_\varepsilon)^{-\frac{1}{2}}\right)\right\|_\A\\
&\leq \|a-\varphi_\varepsilon\circ\psi_\varepsilon(a)\|_\A \\
&\quad + \left\|\unit_\varepsilon - \psi_\varepsilon(\unit_\A)^{-\frac{1}{2}}\right\|_{\B_\varepsilon}\|a\|_\A\left(1 + \|\psi_\varepsilon(\unit_\A)^{-\frac{1}{2}}\|_{\B_\varepsilon}\right)\\
&\quad + \left\|\unit_\A - \varphi_\varepsilon(\unit_\varepsilon)^{-\frac{1}{2}}\right\|_\A \|\psi_\varepsilon(\unit_\A)^{-\frac{1}{2}}\|_{\B_\varepsilon}^2\|a\|_\A\left(1 + \|\varphi_\varepsilon(\unit_\varepsilon)^{-\frac{1}{2}}\|_\A\right)\text{.}
\end{split}
\end{equation*}

Since $\lim_{\varepsilon\rightarrow 0} \|\varphi_\varepsilon(\unit_\varepsilon)-\unit_\A\|_\A = 0$, we check again easily that $\lim_{\varepsilon\rightarrow 0} \|\varphi_\varepsilon(\unit_\varepsilon)^{-\frac{1}{2}}-\unit_\A\|_\A = 0$, and thus, since $\lim_{\varepsilon\rightarrow 0}\|a-\varphi_\varepsilon\circ\psi_\varepsilon(a)\|_\A = 0$, we conclude:
\begin{equation*}
\begin{split}
\lim_{\varepsilon\rightarrow 0}\| a - \theta_\varepsilon\circ\varsigma_\varepsilon(a) \|_\A = 0\text{.}
\end{split}
\end{equation*}

Now, if $c,c'\in\alg{F}$, then we have:
\begin{equation*}
\begin{split}
0 &\leq \|\varsigma_\varepsilon\left(\Jordan{c}{c'}\right) - \Jordan{\varsigma_\varepsilon(c)}{\varsigma_\varepsilon(c')}\|_{\B_\varepsilon} \\
&\leq \frac{1}{2}\left(\|\varsigma_\varepsilon(c)\varsigma_\varepsilon(c') - \varsigma_\varepsilon(cc')\|_{\B_\varepsilon} + \|\varsigma_\varepsilon(c')\varsigma_\varepsilon(c) - \varsigma_\varepsilon(c'c)\|_{\B_\varepsilon}\right)
\end{split}
\end{equation*}
and thus, for all $c,c'\in\alg{F}$:
\begin{equation*}
\lim_{\varepsilon\rightarrow 0} \|\varsigma_\varepsilon\left(\Jordan{c}{c'}\right) - \Jordan{\varsigma_\varepsilon(c)}{\varsigma_\varepsilon(c')}\|_{\B_\varepsilon} = 0 \text{.}
\end{equation*}

The same holds for the Lie product.

This concludes our proof.
\end{proof}

If $(\A,\Lip)$ is a {\gQqcms}, with $\A$ a pseudo-diagonal C*-algebra, then Definition (\ref{pseudo-diagonal-def}) provides some finite dimensional C*-algebraic approximations. The question, of course, is how to construct Lip-norms on these finite dimensional approximations. The main difficulty is to obtain Lip-norms with the Leibniz property. We are, in fact, only able to construct Lip-norms with the quasi-Leibniz property, albeit arbitrarily close to the Leibniz property.

This construction is given in the following lemma:

\begin{lemma}\label{approx-lemma}
Let $C\geq 1$ and $D \geq 0$. Let $(\A,\Lip_\A)$ be a {\Qqcms{C,D}}, and let $\varepsilon > 0$. Assume that we are given a $\varepsilon^2$-dense finite subset $\alg{F}$ of:
\begin{equation*}
\left\{ a \in \sa{\A} : \Lip_\A(a)\leq 1\text{ and }\mu(a) = 0 \right\}
\end{equation*}
for some $\mu\in\StateSpace(\A)$, as well as a finite dimensional C*-algebra $\B$ and two unital, positive linear maps $\varphi : \B \rightarrow \A$ and $\psi : \A\rightarrow\B$ such that:
\begin{enumerate}
\item for all $a\in\alg{F}$, we have $\|a-\varphi\circ\psi(a)\|_\A \leq \varepsilon^2$,
\item for all $a,b\in\alg{F}$, we have $\|\Jordan{\psi(a)}{\psi(b)} - \psi\left(\Jordan{a}{b}\right)\|_\B \leq \varepsilon^2$,
\item for all $a,b\in\alg{F}$, we have $\|\Lie{\psi(a)}{\psi(b)} - \psi\left(\Lie{a}{b}\right)\|_\B \leq \varepsilon^2$.
\end{enumerate}

We then define:
\begin{equation*}
\alg{L} = \left\{ b \in \sa{\B} : \exists a\in\sa{\A} \quad \Lip_\A(a)\leq 1\text{ and }\|b-\psi(a)\|_{\B} \leq\varepsilon \right\} \text{.}
\end{equation*}
Then $\alg{L}$ is a balanced convex set and the associated Minkowsky functional $\Lip$ is a:
\begin{equation*}
\left[C(1+2\varepsilon), C(2\varepsilon + 10\varepsilon^2+12\varepsilon^3) + D\right]
\end{equation*}
quasi-Leibniz Lip-norm on $\B$. 
\end{lemma}

\begin{proof}

We split our proof in a few steps for clarity.

\setcounter{step}{0}
\begin{step}
We first check that we may replace $\alg{F}$ with $\alg{F}+\R\unit_\A$ in our lemma, at no cost.
\end{step}

Our first observation is that if $t, u\in\R$ then, and for all $c, c' \in\alg{F}$, we have:
\begin{equation*}
\begin{split}
\|\psi(c+t\unit_\A)\psi(c'+u\unit_\A)&-\psi((c+t\unit_\A)(c'+u\unit_\A))\|_\B \\
&\leq \|\psi(c)\psi(c')-\psi(cc')\|_\B + \|t\psi(c')-t\psi(c')\|_\B \\
&\quad + \|u\psi(c)-u\psi(c)\|_\B + \|(tu-tu)\unit_\A\|_\B\\
&= \|\psi(c)\psi(c')-\psi(cc')\|_\B\text{.}
\end{split}
\end{equation*}
Consequently, since $c,c' \in \sa{\A}$, $t,u \in \R$ and $\psi$ is positive (hence star-preserving), we have:
\begin{equation*}
\|\Jordan{\psi(c+t\unit_\A)}{\psi(c'+u\unit_\A)} - \psi\left(\Jordan{\left(c+t\unit_\A\right)}{\left(c'+u\unit_\A\right)}\right)\|_\B \leq \varepsilon^2
\end{equation*}
and
\begin{equation*}
\|\Lie{\psi(c+t\unit_\A)}{\psi(c'+u\unit_\A)} - \psi\left(\Lie{c+t\unit_\A}{c'+u\unit_\A}\right)\|_\B \leq \varepsilon^2\text{.}
\end{equation*}

We also have:
\begin{equation*}
\|c+t\unit_\A - \varphi\circ\psi(c+t\unit_\A)\|_\A  = \|c-\varphi\circ\psi(c)\|_\A\leq\varepsilon^2\text{.}
\end{equation*}
Hence, we may replace $\alg{F}$ with $\alg{G} = \left(\alg{F}\cup\{0\}\right) + \R\unit_\A$, while Assumptions (1), (2) and (3) of our Lemma remain satisfied.

Now, let $a\in\sa{\A}$ with $\Lip_\A(a)\leq 1$. Then $\Lip_\A(a-\mu(a)\unit_\A) \leq 1$ and $\mu(a-\mu(a)\unit_\A) = 0$. Thus there exists $c\in\alg{F}$ such that $\|(a-\mu(a)\unit_\A) - c\|_\A \leq\varepsilon^2$. Hence, if $d=c+\mu(a)\unit_\A$ then $d\in\alg{G}$ and of course $\|a-d\|_\A\leq\varepsilon^2$.

\begin{step}
With these observations in mind, let us study the geometry of the set $\alg{L}$. We begin by showing that $\alg{L}$ is a balanced, convex set, and that $\Lip$ is finite on $\B$.
\end{step}

If $b,b' \in \alg{L}$ then there exists $a,a'\in \sa{\A}$ with $\|b-\psi(a)\|_\B\leq \varepsilon$ and $\|b'-\psi(a')\|_\B\leq\varepsilon$ while $\Lip_\A(a)\leq 1$ and $\Lip_\A(a')\leq 1$. Thus, for any $t\in[0,1]$, we have $\Lip_\A(ta+(1-t)a)\leq 1$ and $\|(tb+(1-t)b') - (ta+(1-t)a')\|_\B\leq \varepsilon$. Thus $tb+(1-t)b' \in \alg{L}$, i.e $\alg{L}$ is convex. 

Moreover, for all $t\in[-1,1]$, we have $\Lip_\A(ta)\leq |t| \leq  1$ and $\|tb-\psi(ta)\|_\B\leq t\varepsilon \leq\varepsilon$, so $\alg{L}$ is balanced.

Since $\Lip_\A(0) = 0$, we have $0 \in \alg{L}$, and thus by construction of $\alg{L}$, the closed ball of center $0$ and radius $\varepsilon > 0$ lies in $\alg{L}$. This implies in turn that the Minkowsky functional of $\alg{L}$ is finite on $\B$.

\begin{step}
We now check that $\Lip(b) = 0$ if and only if $b$ is a scalar multiple of the unit of $\B$.
\end{step}

 First, let $t \in \R$. Then $\|t\unit_\B - \psi(t\unit_\A)\|_\B = 0$ and $\Lip_\A(t\unit_\A) = 0$ so $t\unit_\B \in \alg{L}$. Thus $\Lip(t \unit_\B) \leq 1$ and thus, as $t\in\R$ is arbitrary, then $\Lip(\unit_\B) = 0$.  

Assume that $\Lip(b) = 0$ for some $b\in \sa{\B}$. Then $\Lip(tb) = 0$ for all $t\in \R$, so $tb \in \alg{L}$ for all $t\in\R$. Thus for all $t\in \R$ there exists $a_t \in \sa{\A}$ such that $\Lip_\A(a_t)\leq 1$ and:
\begin{equation*}
\|tb - \psi(a_t)\|_\B \leq \varepsilon\text{.}
\end{equation*}
Thus, if $t > 0$, then  $\left\|b - \psi\left(\frac{1}{t}a_t\right)\right\|_\B \leq \frac{1}{t}\varepsilon$. Let $a'_t = \frac{1}{t}a_t$ for all $t > 0$. Now $\Lip_\A(a'_t)\leq t^{-1}$.

Now, for all $n\in\N\setminus\{0\}$, there exists $c_n \in \alg{G}$ with $\Lip_\A(c_n)\leq 1$ and $\|a_n-c_n\|_\A \leq \varepsilon^2$. Let $d_n = \frac{1}{n} c_n$ for all $n\in\N, n > 0$. Thus:
\begin{equation*}
\begin{split}
\|b-\psi(d_n)\|_\B &\leq \|b-\psi(a_n')\|_\B+\|\psi(a_n' - d_n)\|_\B\\
&\leq \frac{\varepsilon + \varepsilon^2}{n}.
\end{split}
\end{equation*}
and thus, we have:
\begin{equation*}
\|a_n'\|_\A \leq \|d_n\|_\A + \frac{\varepsilon^2}{n} \leq \|\psi(d_n)\|_\B + 2\frac{\varepsilon^2}{n} \leq \frac{\varepsilon + 3\varepsilon^2}{n} + \|b\|_\B
\end{equation*}
by Remark (\ref{useful-rmk}), while $\Lip(a_n')\leq \frac{1}{n}$. Therefore, the sequence $(a_n')_{n\in\N, n>0}$ lies in the compact set $\{ c \in \sa{\A} : \Lip_\A(c) \leq 1, \|c\|_\A\leq \|b\|_\B+\varepsilon +3\varepsilon^2 \}$ (using Theorem (\ref{Rieffel-thm})). So we can extract a subsequence of $(a_n')_{n\in\N, n>0}$ converging to some $a$. Now by continuity, $\|b-\psi(a)\|_\B=0$ while, by lower semi-continuity, $\Lip_\A(a) = 0$. Thus $a=k\unit_\A$ for some $k\in\R$ since $(\A,\Lip_\A)$ is a Lipschitz pair. Thus as $\psi$ is unital, we conclude that $b = k\unit_\B$.

\begin{step}
We prove that for any $\nu\in\StateSpace(\B)$, we have $\alg{L} = \{b\in\alg{L} :\nu(b) = 0 \} + \R\unit_\B$.
\end{step}

If $b \in \alg{L}$ then there exists $a\in\sa{\A}$ with $\Lip_\A(a)\leq 1$ and $\|\psi(a)-b\|_\B\leq \varepsilon$. Thus for all $t \in \R$, $\Lip_\A(a- t\unit_\A) = \Lip_\A(a) \leq 1$ since $\Lip_\A(\unit_\A) = 0$ and $\Lip_\A$ is a seminorm, and then:
\begin{equation*}
\|(b-t\unit_\B) - \psi(a-t\unit_\A)\|_\B = \|b-\psi(a)\|_\B \leq \varepsilon\text{.}
\end{equation*}
Consequently, $b-t\unit_\B \in \alg{L}$. From this, we conclude that:
\begin{itemize}
\item if $b\in\alg{L}$ then $b - \nu(b) \unit_\B \in \alg{L}$. Since $\nu(b-\nu(b)\unit_\B) = 0$, and $b = (b-\nu(b)\unit_\B) + \nu(b)\unit_\B$, we have $b\in \{ d\in\alg{L} : \nu(d) = 0 \} + \R\unit_\B$.
\item if $b = d  +t\unit_\B$ with $d\in \alg{L}$ and $\nu(d) = 0$, then $b \in \alg{L}$.
\end{itemize}

This proves our step.

\begin{step}
Fix $\nu \in\StateSpace(\B)$. We now prove that $\alg{L}$ is closed and $\{b\in\alg{L}: \nu(b) = 0 \}$ is compact.
\end{step}

Let $(b_n)_{n\in\N}$ be a sequence in $\alg{L}$ with $\nu(b_n) = 0$. Then for all $n\in\N$, let $a_n \in \sa{\A}$ with $\Lip_\A(a_n)\leq 1$ and $\|\psi(a_n)-b_n\|_\B\leq\varepsilon$. Let $\xi = \nu\circ\psi \in \StateSpace(\A)$ (note that $\psi$ is unital, positive, linear). 

Let $n\in\N$. We first note that:
\begin{equation*}
|\xi(a_n)|\leq |\xi(a_n)-\nu(b_n)|+|\nu(b_n)| \leq |\nu(\psi(a_n)-b_n)|\leq \varepsilon\text{.}
\end{equation*} 

Hence, for all $\eta\in\StateSpace(\A)$, we have:
\begin{equation*}
|\eta(a_n)| \leq |\eta(a_n) - \xi(a_n)| + |\xi(a_n)| \leq \Kantorovich{\Lip_\A}(\eta,\xi) + \varepsilon
\end{equation*}
since $\Lip_\A(a_n)\leq 1$; since $a_n\in\sa{\A}$ we conclude:
\begin{equation*}
\|a_n\|_\A \leq \diam{\StateSpace(\A)}{\Kantorovich{\Lip_\A}} + \varepsilon\text{.}
\end{equation*}

We thus observe that for all $n\in\N, n > 0$:
\begin{equation*}
\|b_n\|_\B \leq \|a_n\|_\A + \varepsilon \leq \diam{\StateSpace(\A)}{\Kantorovich{\Lip_\A}} + 2\varepsilon\text{.}
\end{equation*}
As $\B$ is finite dimensional, the bounded sequence $(b_n)_{n\in\N}$ admits a convergent subsequence $(b_{g(n)})_{n\in\N}$ with limit denoted by $b\in\sa{\B}$. By continuity of $\nu$ we have $\nu(b) = 0$.

We note, as a digression, that we also can conclude that $\diam{\StateSpace(\B)}{\Kantorovich{\Lip}}<\infty$, though we will conclude a stronger fact shortly.

Indeed, we have now proven that for all $n\in\N$, we have $\Lip_\A(a_{g(n)})\leq 1$ and $\|a_{g(n)}\|_\A\leq K$ for some fixed $K>0$. The set $\{a\in\sa{\A}:\Lip_\A(a)\leq 1, \|a\|_\A\leq K\}$ is norm compact since $(\A,\Lip_\A)$ is a {\gQqcms} by Theorem (\ref{Rieffel-thm}). Thus, there exists a convergent subsequence $(a_{g\circ f(n)})_{n\in\N}$ of $(a_{g(n)})_{n\in\N}$ which converges in norm to some $a\in \sa{\A}$, with $\Lip_\A(a)\leq 1$, as $\Lip_\A$ is a lower semi-continuous Lip-norm. 

Thus by continuity, $\|b-\psi(a)\|_\B\leq\varepsilon$ and thus $b\in \alg{L}$. Hence, we have shown that any sequence of elements in $\{ b\in\alg{L} : \nu(b) = 0\}$ has a convergent subsequence with limit in the same set. Therefore, $\{b \in \alg{L} : \nu(b) = 0\}$ is norm compact.

Thus, $\alg{L} = \{b\in\alg{L} : \nu(b) = 0 \} + \R\unit_\B$ is closed. Thus $\Lip$ is a lower semi-continuous Lip-norm on $\B$ by Theorem (\ref{Rieffel-thm}).

\begin{step}
It remains to study the quasi-Leibniz property of $\Lip$.
\end{step}

Let $b,b' \in \alg{L}$. There exist $a,a' \in \sa{\A}$ with $\Lip_\A(a)\leq 1$, $\Lip_\A(a')\leq 1$ while $\|\psi(a)-b\|_\B\leq\varepsilon$ and $\|\psi(a')-b'\|_\B \leq\varepsilon$. Then there exists $c,c' \in \alg{G}$ with $\|a-c\|_\A\leq\varepsilon^2$ and $\|a'-c'\|_\A\leq\varepsilon^2$, as explained in Step 1.

Now:
\begin{equation*}
\|\psi(a)\|_\B\leq \|b\|_\B+\varepsilon \text{ and }\|\psi(c)\|_\B \leq \|\psi(c)-\psi(a)\|_\B+\|\psi(a)\|_\B \leq \|b\|_\B+\varepsilon + \varepsilon^2\text{.}
\end{equation*}

Our goal is now to find an upper bound for $\|\Jordan{b}{b'}-\psi\left(\Jordan{a}{a'}\right)\|_\B$. Thus, as $a,a',b,b'$ are all self-adjoint and $\psi$ preserves the star operation:
\begin{align}\label{term0}
\|\Jordan{b}{b'} &- \psi\left(\Jordan{a}{a'}\right)\|_\B \\
&\leq \|\Jordan{b}{b'}-\Jordan{\psi(a)}{\psi(a')}\|_\B+\|\Jordan{\psi(a)}{\psi(a')}-\psi\left(\Jordan{a}{a'}\right)\|_\B\\
&\leq \|b\|_\B \|b'-\psi(a')\|_\B \label{term1} \\ 
&+ \|\psi(a')\|_\B \|b-\psi(a)\|_\B \label{term2} \\
&+ \|\Jordan{\psi(a)}{\psi(a')}-\Jordan{\psi(c)}{\psi(a')}\|_\B \label{term3} \\
&+\|\Jordan{\psi(c)}{\psi(a')}-\Jordan{\psi(c)}{\psi(c')}\|_\B \label{term4} \\
&+ \|\Jordan{\psi(c)}{\psi(c')}-\psi\left(\Jordan{c}{c'}\right)\|_\B \label{term5} \\ 
&+ \|\psi\left(\Jordan{c}{c'}\right)-\psi\left(\Jordan{a}{a'}\right)\|_\B \text{.} \label{term6}
\end{align}

We now provide a bound for each of the Terms (\ref{term1}--\ref{term6}). By construction, $\|b'-\psi(a')\|_\B \leq \varepsilon$, so Term (\ref{term1}) is dominated by $\|b\|_\B\varepsilon$. Note that Term (\ref{term1}) holds because:
\begin{equation*}
\begin{split}
\|\Jordan{b}{b'} - \Jordan{\psi(a)}{\psi(a')}\|_\B &= \frac{1}{2}\left(\|bb' - \psi(a)\psi(a')\|_\B + \|b'b - \psi(a')\psi(a)\|_\B\right) \\
&= \frac{1}{2}\left(\|bb' - \psi(a)\psi(a')\|_\B + \|(b b' - \psi(a)\psi(a'))^\ast\|_\B\right) \\
&= \|bb'-\psi(a)\psi(a')\|_\B\text{.}
\end{split}
\end{equation*}

Similarly, since $\|\psi(a')\|_\B \leq \|b'\|_\B + \varepsilon$ and $\|\psi(a)-b\|_\B\leq\varepsilon$, we conclude that Terms (\ref{term2}) is dominated by $\left(\|b'\|_\B+\varepsilon\right)\varepsilon$.

Now:
\begin{equation*}
\|\Jordan{\psi(a)}{\psi(a')}-\Jordan{\psi(c)}{\psi(a')}\|_\B \leq \|\psi(a)-\psi(c)\|_\B \|\psi(a')\|_\B \leq \varepsilon^2\left(\|b'\|_\B+\varepsilon\right)
\end{equation*}
which gives us our bound for Term (\ref{term3}). 

We derive a similar upper bound on Term (\ref{term4}):
\begin{multline*}
\|\Jordan{\psi(c)}{\psi(a')}-\Jordan{\psi(c)}{\psi(c')}\|_\B \leq \|\psi(a')-\psi(c')\|_\B \|\psi(c)\|_\B \\ \leq \varepsilon^2\left(\|b\|_\B+\varepsilon+\varepsilon^2\right)\text{.}
\end{multline*}
By our choice of $\psi$, Term (\ref{term5}) is bounded above by $\varepsilon^2$. It is the appearance of this term, and the quasi-multiplicative property of $\psi$, which motivates our computation.

Last, concerning of Term (\ref{term6}), we first note that $\|a\|_\A\leq\|c\|_\A + \varepsilon^2$ while $\|c\|_\A \leq \|\psi(c)\|_\B + \varepsilon^2$ by Remark (\ref{useful-rmk}). Thus:
\begin{equation*}
\|a\|_\A \leq \|b\|_\B+\varepsilon+3\varepsilon^2\text{,}
\end{equation*}
while similarly, $\|c'\|_\A \leq \|\psi(c')\|+\varepsilon^2 \leq \|b'\| + \varepsilon + 2\varepsilon^2$. Thus:
\begin{equation*}
\|aa'-cc'\|_\A \leq \|a\|_\A\|a'-c'\|_\A + \|c'\|_\A\|a-c\|_\A \leq \varepsilon^2\left(\|b\|_\B + \|b'\|_\B + 2\varepsilon + 5\varepsilon^2\right)\text{.}
\end{equation*}
Consequently:
\begin{equation*}
\|\psi(\Jordan{a}{a'}-\Jordan{c}{c'})\|_\B \leq \|\Jordan{a}{a'}-\Jordan{c}{c'}\|_\A \leq  \varepsilon^2\left(\|b\|_\B + \|b'\|_\B + 2\varepsilon + 5\varepsilon^2\right)\text{.}
\end{equation*}

We thus obtain the following estimate:
\begin{multline}
\|\Jordan{b}{b'}-\psi(\Jordan{a}{a'})\|_\B \leq \|b\|_\B\varepsilon + \left(\|b'\|_\B+\varepsilon\right)\varepsilon \\
+ \varepsilon^2\left(\|b'\|_\B+\varepsilon\right) + \varepsilon^2\left(\|b\|_\B+\varepsilon+\varepsilon^2\right)\\
+ \varepsilon^2 + \varepsilon^2\left(\|b\|_\B + \|b'\|_\B + 2\varepsilon + 5\varepsilon^2\right)\text{,}
\end{multline}
which simplifies to:
\begin{equation}\label{approx-lemma-eq1}
\|\Jordan{b}{b'}-\psi\left(\Jordan{a}{a'}\right)\|_\B \leq \varepsilon\left( (1+2\varepsilon)\left(\|b\|_\B + \|b'\|_\B\right) + 2\varepsilon + 4\varepsilon^2 + 6\varepsilon^3 \right)\text{.}
\end{equation}

We now compute an upper bound on the Lip-norm of the Jordan product $\Jordan{a}{a'}$, using the $(C,D)$-quasi-Leibniz property of $\Lip_\A$, and using $\|b\|_\B$ and $\|b'\|_\B$:

\begin{equation}\label{approx-lemma-eq3}
\begin{split}
\Lip_\A\left(\Jordan{a}{a'}\right) &\leq C\left(\|a\|_\A + \|a'\|_\A\right) + D \\
&\leq C \left( \|b\|_\B + \|b'\|_\B + 2\varepsilon + 6\varepsilon^2 \right) + D\text{.}
\end{split}
\end{equation}

Let:
\begin{equation*}
\begin{split}
\omega &= C(1+2\varepsilon) \left(\|b\|_\B + \|b'\|_\B + 2\varepsilon + 6\varepsilon^2\right) + D \\
&= C(1+2\varepsilon)\left(\|b\|_\B + \|b'\|_\B\right) + C(2\varepsilon + 10\varepsilon^2 + 12\varepsilon^3) + D \text{.}
\end{split}
\end{equation*}

Trivially, from Expression (\ref{approx-lemma-eq1}), we have:
\begin{equation*}
\|\Jordan{b}{b'}-\psi\left(\Jordan{a}{a'}\right)\|_\B \leq \varepsilon\left( C(1+2\varepsilon)\left(\|b\|_\B + \|b'\|_\B\right) + C(2\varepsilon + 10\varepsilon^2 + 12\varepsilon^3) + D\right) = \varepsilon\omega \text{,}
\end{equation*}
since $C\geq 1$, $D\geq 0$, so:
\begin{equation}\label{approx-lemma-eq2}
\left\|\frac{1}{\omega}\Jordan{b}{b'} - \psi\left(\frac{1}{\omega}\Jordan{a}{a'}\right)\right\|_\B \leq \varepsilon\text{.}
\end{equation}

Similarly, from Expression (\ref{approx-lemma-eq3}), we have:
\begin{equation}\label{approx-lemma-eq4}
\Lip_\A\left(\frac{1}{\omega} \Jordan{a}{a'} \right) \leq 1\text{.}
\end{equation}

\bigskip

Inequalities (\ref{approx-lemma-eq2}) and (\ref{approx-lemma-eq4}) together prove that $\frac{1}{\omega} \Jordan{b}{b'}\in \alg{L}$. Thus, to summarize the heavy computation above, we conclude that:
\begin{equation*}
\Jordan{b}{b'} \in \left[C\left(1+2\varepsilon\right)\left(\|b\|_\B + \|b'\|_\B\right) + \left(C(2\varepsilon + 10\varepsilon^2 + 12\varepsilon^3) + D\right)\right] \alg{L}\text{.}
\end{equation*}

The same computations hold for the Lie product as well.

Thus Lemma (\ref{quasi-Leibniz-lemma}) proves that $\Lip$ satisfies the $\left(C(1+2\varepsilon), C(2\varepsilon + 10\varepsilon^2+12\varepsilon^3) + D\right)$-quasi-Leibniz identity.
\end{proof}

We thus arrive at the following finite dimensional approximation result which applies to a large class of {\gQqcms s}, including in particular nuclear, quasi-diagonal {\Lqcms s}. This result employs the generalized dual propinquity adapted to {\gQqcms s}, as Lemma (\ref{approx-lemma}) only provides us with such finite dimensional approximations, although the departure from the Leibniz property can be controlled to be made arbitrarily small.

\begin{theorem}
Let $C\geq 1$, $D\geq 0$ and Let $(\A,\Lip_\A)$ be a {\Qqcms{(C,D)}} where $\A$ is a unital pseudo-diagonal C*-algebra. Then $(\A,\Lip_\A)$ is the limit of finite dimensional $(C(1+\varepsilon) , D+ \varepsilon)$-{\gQqcms s} for the $(C(1+\varepsilon),D+\varepsilon)$-dual propinquity, for any $\varepsilon > 0$.
\end{theorem}

\begin{proof}
Let $\varepsilon > 0$. Let $\mu\in\StateSpace(\A)$ be arbitrary. The set:
\begin{equation*}
\alg{l} = \left\{ a\in\sa{\A} : \Lip_\A(a)\leq 1\text{ and }\mu(a) = 0 \right\}
\end{equation*}
is norm-compact, since $\Lip_\A$ is a lower semi-continuous Lip-norm, by Theorem (\ref{Rieffel-thm}). Thus there exists a $\varepsilon^2$-dense finite subset $\alg{F}$ of $\alg{l}$.

Since $\A$ is pseudo-diagonal, there exists a finite dimensional C*-algebra $\B$ and two unital positive linear maps $\varphi : \B\rightarrow\A$ and $\psi:\A\rightarrow\B$ such that:
\begin{enumerate}
\item for all $a\in\alg{F}$ we have $\|a-\varphi\circ\psi(a)\|_\A\leq \varepsilon^2$,
\item for all $a,b \in \alg{F}$ we have $\|\Jordan{\psi(a)}{\psi(b)} - \psi\left(\Jordan{a}{b}\right)\|_\B \leq \varepsilon^2$,
\item for all $a,b \in \alg{F}$ we have $\|\Lie{\psi(a)}{\psi(b)} - \psi\left(\Lie{a}{b}\right)\|_\B \leq \varepsilon^2$.
\end{enumerate}

Now, as with step 1 of the proof of Lemma (\ref{approx-lemma}), Assertions (1),(2) and (3) above remain valid if $\alg{F}$ is replaced by $\alg{F}_1 = \alg{F} + \R\unit_\A$. Moreover for any $a\in\sa{\A}$ with $\Lip_\A(a)\leq 1$ then there exists $c\in\alg{F}_1$ such that $\|a-c\|_\A\leq\varepsilon^2$. We shall use this henceforth.

By Lemma (\ref{approx-lemma}), if we set:
\begin{equation*}
\alg{L} = \left\{ b \in \sa{\B} : \exists a \in \sa{\A} \quad \|b-\psi(a)\|_\B \leq \varepsilon\text{ and } \Lip_\A(a)\leq 1 \right\}
\end{equation*}
and if $\Lip_\B$ is the associated Minkowsky functional, then $\Lip$ is a:
\begin{equation*}
\left[C(1+2\varepsilon), C(2\varepsilon + 10\varepsilon^2+12\varepsilon^3)+D\right]
\end{equation*}
quasi-Leibniz Lip-norm on $\B$. We now wish to compute the dual propinquity between $(\A,\Lip_\A)$ and $(\B,\Lip_\B)$.

For all $a\in\sa{\A}$ and $b\in\sa{\B}$ we define:

\begin{equation*}
N(a,b) = \frac{1}{\varepsilon}\|\psi(a)-b\|_\B\text{,}
\end{equation*}
and
\begin{equation*}
\Lip(a,b) = \max\left\{ \Lip_\A(a), \Lip_\B(b), N(a,b)\right\}\text{.}
\end{equation*}

We note that $N$ is norm continuous on $\sa{\A\oplus\B}$, while $N(\unit_\A,\unit_\B) = 0$ yet $N(\unit_\A,0) \not= 0$ since $\psi$ is unital.

Now, let $a\in\sa{\A}$ with $\Lip_\A(a) = 1$. Then since $\|\psi(a)-\psi(a)\|_\B = 0$, we conclude that $\Lip_\B(\psi(a)) \leq 1$ and $N(a,\psi(a)) = 0$. Note that $\Lip(a,\psi(a)) = \Lip_\A(a) = 1$ in this case.

Second, let $b\in\sa{\B}$ with $\Lip_\B(b) = 1$. Then there exists $a\in\sa{\A}$ with $\Lip_\A(a)\leq 1$ and $N(a,b)\leq 1$ by definition of $N$ and $\Lip_\B$. Again note that $\Lip(a,b) = \Lip_\B(b) = 1$.

Consequently, $N$ is a bridge in the sense of \cite{Rieffel00}. We conclude immediately from the above computation that in fact, the quotient of $\Lip$ on $\sa{\A}$ is $\Lip_\A$ and the quotient of $\Lip$ on $\sa{\B}$ is $\Lip_\B$.

However, to conclude that we have constructed a tunnel for the propinquity, we must also study the quasi-Leibniz property for $\Lip$. 

Let $a,a' \in\dom{\Lip_\A}$ and $b,b'\in\dom{\Lip_\B}$. Let $t = \max\{\Lip_\A(a),\Lip_\B(b)\}$ and $u=\max\{\Lip_\A(a'),\Lip_\B(b')\}$.

Assume first $t\leq 1$ and $u\leq 1$. Thus, there exists $c,c' \in \alg{F}_1$ such that $\|a-c\|_\A\leq\varepsilon^2$ and $\|a'-c'\|_\A\leq\varepsilon^2$ again. Thus, similarly to our computation in the proof of Lemma (\ref{approx-lemma}), we have:

\begin{align}
\|bb' &- \psi\left(aa'\right)\|_\B \nonumber \\
&\leq \|b\|_\B \|b'-\psi(a')\|_\B \label{term1a} \\ 
&+ \|\psi(a')\|_\B \|b-\psi(a)\|_\B \label{term2a} \\
&+ \|\psi(a)\psi(a')-\psi(c)\psi(a')\|_\B \label{term3a} \\
&+\|\psi(c)\psi(a')-\psi(c)\psi(c')\|_\B \label{term4a} \\
&+ \|\psi(c)\psi(c')-\psi\left(cc'\right)\|_\B \label{term5a} \\ 
&+ \|\psi\left(cc'\right)-\psi\left(aa'\right)\|_\B \text{.} \label{term6a}
\end{align}

\begin{itemize}
\item Term (\ref{term1a}) is bounded above by $\|b\|_\B (\varepsilon N(a',b'))$. 
\item Term (\ref{term2a}) is bounded above by $\|a'\|_\A (\varepsilon N(a,b))$.
\item Term (\ref{term3a}) is bounded above by $\|a'\|\varepsilon^2$.
\item Term (\ref{term4a}) is bounded above by $\left(\|a\|+\varepsilon^2\right)\varepsilon^2$.
\item Term (\ref{term5a}) is bounded above by $\varepsilon^2$.
\item Term (\ref{term6a}) is bounded above by $\|cc'-aa'\|_\A$ which in turn, is bounded by $\left(\|a\|_\A + \|a'\|_\A + \varepsilon^2\right)\varepsilon^2$.
\end{itemize}

Let us now assume simply that $t, u > 0$. Thus, we get from our estimates above:
\begin{equation*}
\begin{split}
\frac{1}{tu} N\left(aa',bb'\right) &= N\left(\frac{1}{t}a\frac{1}{u}a', \frac{1}{t}b\frac{1}{u}b'\right)\\
&\leq \left\|\frac{1}{u}b'\right\|_\B N\left(\frac{1}{t}a,\frac{1}{t}b\right)  + \left\|\frac{1}{t}a\right\|_\A N\left(\frac{1}{u}a', \frac{1}{u}b'\right) \\
&\quad + 2\left(\left\|\frac{1}{t}a\right\|_\A + \left\|\frac{1}{u}a'\right\|_\A\right)\varepsilon\\
&\quad + \varepsilon + 2\varepsilon^3\text{.}
\end{split}
\end{equation*}

Hence:
\begin{multline*}
N(aa',bb') \leq \|b'\|_\B N(a,b) + \|a\|_\A N(a',b') \\ + 2\|a\|_\A u \varepsilon + 2\|a'\|_\A t \varepsilon + \varepsilon t u + 2 \varepsilon^3 t u\text{.}
\end{multline*}

Since $\|(a,b)\|_{\A\oplus\B}=\max\{\|a\|_\A,\|b\|_\B\}$, $\|(a',b')\|_{\A\oplus\B} =\max\{\|a'\|_\A,\|b'\|_\B\}$ while $\max\{t,N(a,b)\} = \Lip(a,b)$ and $\max\{u, N(a',b')\} = \Lip(a',b')$, we get:
\begin{multline}\label{final-thm-eq1}
N(aa',bb') \leq (\|(a',b')\|_{\A\oplus\B}(1+2\varepsilon)) \Lip(a,b) \\
+ (1+2\varepsilon)\|(a,b)\|_{\A\oplus\B}\Lip(a',b') + \left(\varepsilon + 2\varepsilon^3\right)\Lip(a,b)\Lip(a',b')\text{.}
\end{multline}

If, instead, $t=0$, then $a$ and $b$ are scalar multiple of the identities in $\A$ and $\B$ respectively, and Inequality (\ref{final-thm-eq1}) trivially holds. The same is true if $u = 0$. In conclusion, Inequality (\ref{final-thm-eq1}) holds for all $a,a'\in\sa{\A}$ and $b,b'\in\sa{\B}$.

Since we must use the dual propinquity adapted, in particular, to the $(C(1+2\varepsilon), C(2\varepsilon + 10\varepsilon^2 + 12\varepsilon^3)+D)$-{\gQqcms} $(\B,\Lip_\B)$, we simply note the weaker estimate:

\begin{multline*}
N(aa',bb') \leq C(1+2\varepsilon)\left(\|(a,b)\|_{\A\oplus\B} \Lip(a',b') + \|(a',b')\|_{\A\oplus\B}\Lip(a,b)\right) \\ 
+ \left(C(2\varepsilon + 10\varepsilon^2 + 12\varepsilon^3) + D\right)\Lip(a',b')\Lip(a,b)\text{.}
\end{multline*}

It is then straightforward that for $a,a',b,b'$ are self-adjoint, we have:
\begin{multline*}
\max\left\{N\left(\Jordan{a}{a'},\Jordan{b}{b'}\right), N\left(\Lie{a}{a'},\Lie{b}{b' }\right)\right\} \leq \\
C(1+2\varepsilon)\left(\|(a,b)\|_{\A\oplus\B}\Lip(a',b') + \|(a',b')\|_{\A\oplus\B }\Lip(a,b)\right) \\ 
+ \left(C(2\varepsilon + 10\varepsilon^2 + 12\varepsilon^3) + D\right)\Lip(a',b')\Lip(a,b)\text{.}
\end{multline*}

From this we conclude that, if $\pi_\A:\A\oplus\B \twoheadrightarrow \A$ and $\pi_\B: \A\oplus\B \twoheadrightarrow \B$ are the canonical surjections, then $\tau = (\A\oplus\B, \Lip,\pi_\A,\pi_\B)$ is a $(C(1+2\varepsilon), C(2\varepsilon + 10\varepsilon^2+12\varepsilon^3)+D)$-tunnel in the sense of Definition (\ref{tunnel-def}). 

We can thus compute the length of the tunnel $\tau$, as defined in Definition (\ref{length-def}), to obtain an upper bound for the dual propinquity between $(\A,\Lip_\A)$ and $(\B,\Lip_\B)$.

The depth of $\tau$ is trivially zero, so we simply compute its reach.

If $\theta\in\StateSpace(\B)$ then we set $\varsigma = \theta\circ\psi\in\StateSpace(\A)$ (as $\psi$ is unital, positive linear) and observe that:
\begin{equation*}
\begin{split}
\Kantorovich{\Lip}(\theta,\varsigma) &= \sup\{|\theta(b) - \theta\circ\psi(a)| : \Lip_\A(a)\leq 1, \Lip_\B(b)\leq 1, \|b-\psi(a)\|_\B\leq \varepsilon \}\\
&\leq \varepsilon\text{.}
\end{split}
\end{equation*}

If $\varsigma\in\StateSpace(\A)$, then let $\theta =\varsigma\circ\varphi \in\StateSpace(\B)$ (since $\varphi$ is unital, positive linear). If $N(a,b)\leq 1$ for some $(a,b)\in\sa{\A\oplus\B}$ then $\|b-\psi(a)\|_\B \leq\varepsilon$ and, since $\Lip_\A(a)\leq 1$, there exists $c\in\alg{F}_1$ such that $\|a-c\|_\A\leq\varepsilon^2$, so:
\begin{equation*}
\begin{split}
\|\varphi(b)-a\|_\A &\leq \|\varphi(b)-\varphi\circ\psi(a)\|_\A + \|\varphi\circ\psi(a)-a\|_\A\\
&\leq \|b-\psi(a)\|_\B + \|\varphi\circ\psi(a)-\varphi\circ\psi(c)\|_\A \\
&\quad + \|\varphi\circ\psi(c)-c\|_\A + \|c-a\|_\A\\
&\leq \varepsilon + 3\varepsilon^2\text{.}
\end{split}
\end{equation*}
Consequently, $\Kantorovich{\Lip}(\theta,\varsigma) \leq \varepsilon + 3\varepsilon^2$.

Thus the length of our tunnel is bounded above by $\varepsilon + 3\varepsilon^2$ (and its extent is bounded above by $2(\varepsilon + 3\varepsilon^2)$).

In summary, for all $\varepsilon > 0$, there exists a finite dimensional $(C(1+2\varepsilon), C(2\varepsilon + 10\varepsilon^2+12\varepsilon^3)+D)$-{\gQqcms} $(\B,\Lip_\B)$ such that:
\begin{equation*}
\propinquity{C(1+2\varepsilon),C(2\varepsilon + 10\varepsilon^2 + 12\varepsilon^3)+D}((\A,\Lip_\A),(\B,\Lip_\B)) \leq 2\left(\varepsilon + 3\varepsilon^2\right)\text{.}
\end{equation*}

This proves our theorem. Indeed: fix $\delta > 0$ and let $\varepsilon > 0$ be chosen so that $1+2\varepsilon \leq 1+\delta$ and $C(2\varepsilon + 10\varepsilon^2 + 12\varepsilon^3) \leq \delta$. Then, we have shown that there exists a finite dimensional $(C(1+\delta), D+\delta)$-{\gQqcms} $(\B,\Lip_\B)$ such that $\propinquity{C(1+\delta),D+\delta}((\A,\Lip_\A),(\B,\Lip_\B))\leq 2(\varepsilon + 3\varepsilon^2)$. From this, we conclude that $(\A,\Lip_\A)$ is the limit, for $\propinquity{C(1+\delta,D+\delta)}$, of finite dimensional $(C(1+\delta), D+\delta)$-{\gQqcms s}.
\end{proof}

Now, in particular, we conclude that for any $C>1$ and $D>0$, if $(\A,\Lip_\A)$ is a pseudo-diagonal {\Lqcms}, then $(\A,\Lip_\A)$ is a limit of finite dimensional $(C,D)$-{\gQqcms s} for $\propinquity{C,D}$. Our justification for working with {\gQqcms s} stems from the fact that our proof, in general, does not allow $C=1$ and $D=0$. Of course, \cite{Latremoliere13b}, \cite{Rieffel10c} provide examples of {\Lqcms s} which are limits of finite dimensional {\Lqcms s} for the dual propinquity, so there are known nontrivial examples when this desirable result holds.

\providecommand{\bysame}{\leavevmode\hbox to3em{\hrulefill}\thinspace}
\providecommand{\MR}{\relax\ifhmode\unskip\space\fi MR }
\providecommand{\MRhref}[2]{%
  \href{http://www.ams.org/mathscinet-getitem?mr=#1}{#2}
}
\providecommand{\href}[2]{#2}

\vfill


\begin{thebibliography}{10}

\bibitem{Blackadar91}
{B}. {B}lackadar and {J}. {C}untz, \emph{Differential {B}anach algebra norms
  and smooth subalgebras of {C*--A}lgebras}, Journal of Operator Theory
  \textbf{26} (1991), no.~2, 255--282.

\bibitem{Blackadar97}
{B}. {B}lackadar and {E}. {K}irchberg, \emph{Generalized inductive limits of
  finite dimensional {$C^\ast$}-algebras}, Math. Ann. \textbf{307} (1997),
  343--380.

\bibitem{burago01}
{D}. {B}urago, {Y}. {B}urago, and {S}. {I}vanov, \emph{A course in metric
  geomtry}, Graduate Studies in Mathematics, vol.~33, American Mathematical
  Society, 2001.

\bibitem{Connes89}
A.~{C}onnes, \emph{Compact metric spaces, {F}redholm modules and
  hyperfiniteness}, Ergodic Theory and Dynamical Systems \textbf{9} (1989),
  no.~2, 207--220.

\bibitem{Connes}
\bysame, \emph{Noncommutative geometry}, Academic Press, San Diego, 1994.

\bibitem{Connes08}
A.~{C}onnes and H.~{M}oscovici, \emph{Type {III} and spectral triples}, Traces
  in Geometry, Number Theory and Quantum Fields, Aspects of Math., vol. E38,
  Springer-Verlag, 2008, pp.~51--71.

\bibitem{Gromov81}
M.~{G}romov, \emph{Groups of polynomial growth and expanding maps},
  Publications math{\'e}matiques de l' {I. H. E. S.} \textbf{53} (1981),
  53--78.

\bibitem{Gromov}
\bysame, \emph{Metric structures for {R}iemannian and non-{R}iemannian spaces},
  Progress in Mathematics, Birkh{\"a}user, 1999.

\bibitem{Guido06}
{D}.~{G}uido and {T}.~{I}sola, \emph{The problem of completeness for Gromov-Hausdorff metrics on $C^\ast$-algebras}, Journal of Funct. Anal. \textbf{233} (2006), 173--205, ArXiv: 0502013.

\bibitem{Kantorovich40}
{L}.~{V}. {K}antorovich, \emph{On one effective method of solving certain
  classes of extremal problems}, Dokl. Akad. Nauk. USSR \textbf{28} (1940),
  212--215.

\bibitem{Kantorovich58}
{L}.~{V}. {K}antorovich and {G}.~{Sh}. {R}ubinstein, \emph{On the space of
  completely additive functions}, Vestnik Leningrad Univ., Ser. Mat. Mekh. i
  Astron. \textbf{13} (1958), no.~7, 52--59, In Russian.

\bibitem{kerr02}
D.~{K}err, \emph{Matricial quantum {G}romov-{H}ausdorff distance}, J. Funct.
  Anal. \textbf{205} (2003), no.~1, 132--167, math.OA/0207282.

\bibitem{kerr09}
{D}. {K}err and {H}. {L}i, \emph{On {G}romov--{H}ausdorff convergence of
  operator metric spaces}, J. Oper. Theory \textbf{1} (2009), no.~1, 83--109.

\bibitem{Latremoliere05}
{F}. {L}atr{\'e}moli{\`e}re, \emph{Approximation of the quantum tori by finite
  quantum tori for the quantum gromov-hausdorff distance}, Journal of Funct.
  Anal. \textbf{223} (2005), 365--395, math.OA/0310214.

\bibitem{Latremoliere05b}
\bysame, \emph{Bounded-lipschitz distances on the state space of a
  {C*}-algebra}, Tawainese Journal of Mathematics \textbf{11} (2007), no.~2,
  447--469, math.OA/0510340.
guido
\bibitem{Latremoliere13}
\bysame, \emph{The {Q}uantum {G}romov-{H}ausdorff {P}ropinquity},
  Trans. Amer. Math. Soc. {\bf 368} (2016) 1, 365--411, ArXiv: 1302.4058.

\bibitem{Latremoliere12b}
\bysame, \emph{Quantum locally compact metric spaces}, Journal of Functional
  Analysis \textbf{264} (2013), no.~1, 362--402, ArXiv: 1208.2398.

\bibitem{Latremoliere13c}
\bysame, \emph{Convergence of fuzzy tori and quantum tori for the
  {G}romov--{H}ausdorff {P}ropinquity: an explicit approach.}, Accepted,
  M{\"u}nster Journal of Mathematics (2014), 41 pages, ArXiv: math/1312.0069.

\bibitem{Latremoliere14b}
\bysame, \emph{A topographic {G}romov-{H}ausdorff hypertopology for quantum
  proper metric spaces}, Submitted (2014), 67 Pages, ArXiv: 1406.0233.

\bibitem{Latremoliere14}
\bysame, \emph{The triangle inequality and the dual {G}romov-{H}ausdorff
  propinquity}, Accepted, Indiana University Mathematics Journal (2015), 16 Pages, ArXiv: 1404.6633.

\bibitem{Latremoliere13b}
\bysame, \emph{The dual {G}romov--{H}ausdorff {P}ropinquity}, Journal de
  Math{\'e}matiques Pures et Appliqu{\'e}es \textbf{103} (2015), no.~2,
  303--351, ArXiv: 1311.0104.

\bibitem{li03}
H.~{L}i, \emph{{$C^\ast$}-algebraic quantum {G}romov-{H}ausdorff distance},
  (2003), ArXiv: math.OA/0312003.

\bibitem{li05}
\bysame, \emph{{$\theta$}-deformations as compact quantum metric spaces}, Comm.
  Math. Phys. \textbf{1} (2005), 213--238, ArXiv: math/OA: 0311500.

\bibitem{li06}
\bysame, \emph{Order-unit quantum {G}romov-{H}ausdorff distance}, J. Funct.
  Anal. \textbf{233} (2006), no.~2, 312--360.

\bibitem{Ozawa05}
{N}. {O}zawa and M.~A. {R}ieffel, \emph{Hyperbolic group {$C\sp\ast$}-algebras
  and free products {$C\sp\ast$}-algebras as compact quantum metric spaces},
  Canad. J. Math. \textbf{57} (2005), 1056--1079, ArXiv: math/0302310.

\bibitem{smith04}
{C}. {P}op and {R}.{R}. {Smith}, \emph{Crossed-products and entropy of
  automorphisms}, Journal of Funct. Anal. \textbf{206} (2004), 210--232.

\bibitem{Rieffel98a}
M.~A. {R}ieffel, \emph{Metrics on states from actions of compact groups},
  Documenta Mathematica \textbf{3} (1998), 215--229, math.OA/9807084.

\bibitem{Rieffel99}
\bysame, \emph{Metrics on state spaces}, Documenta Math. \textbf{4} (1999),
  559--600, math.OA/9906151.

\bibitem{Rieffel02}
\bysame, \emph{Group {$C\sp\ast$}-algebras as compact quantum metric spaces},
  Documenta Mathematica \textbf{7} (2002), 605--651, ArXiv: math/0205195.

\bibitem{Rieffel01}
\bysame, \emph{Matrix algebras converge to the sphere for quantum
  {G}romov--{H}ausdorff distance}, Mem. Amer. Math. Soc. \textbf{168} (2004),
  no.~796, 67--91, math.OA/0108005.

\bibitem{Rieffel05}
\bysame, \emph{Compact quantum metric spaces}, Operator algebras, quantization,
  and noncommutative geometry, Contemporary Math, vol. 365, American
  Mathematical Society, 2005, ArXiv: 0308207, pp.~315--330.

\bibitem{Rieffel06}
\bysame, \emph{Lipschitz extension constants equal projection constants},
  Contemporary Math. \textbf{414} (2006), 147--162, ArXiv: math/0508097.

\bibitem{Rieffel08}
\bysame, \emph{A global view of equivariant vector bundles and {D}irac
  operators on some compact homogenous spaces}, Contemporary Math \textbf{449}
  (2008), 399--415, ArXiv: math/0703496.

\bibitem{Rieffel09}
\bysame, \emph{Distances between matrix alegbras that converge to coadjoint
  orbits}, Proc. Sympos. Pure Math. \textbf{81} (2010), 173--180, ArXiv:
  0910.1968.

\bibitem{Rieffel10c}
\bysame, \emph{{L}eibniz seminorms for "matrix algebras converge to the
  sphere"}, Clay Math. Proc. \textbf{11} (2010), 543--578, ArXiv: 0707.3229.

\bibitem{Rieffel10}
\bysame, \emph{Vector bundles and {G}romov-{H}ausdorff distance}, Journal of
  {K}-theory \textbf{5} (2010), 39--103, ArXiv: math/0608266.

\bibitem{Rieffel11}
\bysame, \emph{Leibniz seminorms and best approximation from
  {$C^\ast$}-subalgebras}, Sci China Math \textbf{54} (2011), no.~11,
  2259--2274, ArXiv: 1008.3773.

\bibitem{Rieffel12}
\bysame, \emph{Standard deviation is a strongly {L}eibniz seminorm}, Submitted
  (2012), 24 pages, ArXiv: 1208.4072.

\bibitem{Rieffel14}
{M}.~{A}. {R}ieffel, \emph{Non-commutative resistance networks}, SIGMA
  \textbf{10} (2014), no.~64, 46 pages, ArXiv: 1401.4622.

\bibitem{Rieffel00}
M.~A. {R}ieffel, \emph{{G}romov-{H}ausdorff distance for quantum metric
  spaces}, Mem. Amer. Math. Soc. \textbf{168} (March 2004), no.~796,
  math.OA/0011063.

\bibitem{Voiculescu91}
{D.} {V}oiculescu, \emph{A note on quasi-diagonal {$C^\ast$}-algebras and
  homotopy}, Duke Math. J. \textbf{62} (1991), no.~2, 267--271.

\end{thebibliography}
\end{document}